\documentclass[a4paper,12pt,leqno]{article}
\usepackage{amssymb}
\usepackage{amsthm}
\usepackage[dvips]{graphicx}
\usepackage{amsmath}
\usepackage{fancyhdr}
\usepackage[all]{xy}

\theoremstyle{plain}
\newtheorem{theorem}{Theorem}[section]
\newtheorem{proposition}[theorem]{Proposition}%[section]
\newtheorem{lemma}[theorem]{Lemma}%[section]
\newtheorem{corollary}[theorem]{Corollary}%[section]
\theoremstyle{remark}
\newtheorem{remark}[theorem]{\sc Remark}%[section]
\newtheorem{example}[theorem]{\sc Example}%[section]
\newtheorem*{notations}{\sc Notations}%[section]
\newtheorem*{acknowledgements}{\sc Acknowledgements}

\renewcommand{\baselinestretch}{1.13}

\newenvironment{reference}[1]{%

\begin{flushleft}\normalsize{\textsc{References}}\end{flushleft}%
\begin{enumerate}\setlength{\itemsep}{-5pt}\small
}{\end{enumerate}}

\makeatletter
\@addtoreset{equation}{section}
\makeatother

\makeatletter
\renewcommand{\section}{%
\@startsection{section}{1}{\z@}%
{3.5ex \@plus -1ex \@minus -.2ex}%
{2.3ex \@plus.2ex}%
{\reset@font\normalsize\scshape}}

\makeatother

\makeatletter
\renewcommand{\subsection}{%
\@startsection{subsection}{2}{\z@}%
{-3.5ex \@plus -1ex \@minus -.2ex}%
{-2.3ex \@plus.2ex}%
{\reset@font\normalsize\scshape}}
\makeatother

\usepackage[OT2,T1]{fontenc}
\DeclareSymbolFont{cyrletters}{OT2}{wncyr}{m}{n}
\DeclareMathSymbol{\Sha}{\mathalpha}{cyrletters}{"58}

\title{\vspace*{-22mm}
\large{\textbf{
On pro-$p$ link groups of number fields \\[1.5mm]
}}
}

\author{
\textsc{\normalsize Yasushi Mizusawa}%\footnote{
}

\date{}

\begin{document}
{\renewcommand{\baselinestretch}{1.05} \maketitle}

\vspace*{-14mm}
\renewcommand{\abstractname}{}
{\renewcommand{\baselinestretch}{1.05}
\begin{abstract}{\small 
\noindent\textsc{Abstract.} 
As an analogue of a link group, 
we consider the Galois group of the maximal pro-$p$-extension of a number field 
with restricted ramification 
which is cyclotomically ramified at $p$, i.e., 
tamely ramified over the intermediate cyclotomic $\mathbb Z_p$-extension of the number field. 
In some basic cases, 
such a pro-$p$ Galois group also has a Koch type presentation 
described by linking numbers and 
mod $2$ Milnor numbers (R\'edei symbols) of primes. 
Then the pro-$2$ Fox derivative yields 
a calculation of Iwasawa polynomials 
analogous to Alexander polynomials. 
}\end{abstract}}

%------------------------------------------------------------------------------
\section{Introduction}

Let $p$ be a fixed prime number. 
We often regard the ring $\mathbb Z_p$ of $p$-adic integers as the additive group. 
Let $k$ be a number field, i.e., an extension of finite degree over the rational number field $\mathbb Q$, 
and let $P$ be the set of all primes of $k$ lying over $p$. 
For an algebraic extension $K/k$ and 
a finite set $S$ of primes of $k$ none of which are complex archimedean, 
we denote by $K_S$ the maximal pro-$p$-extension of $K$ unramified outside primes lying over $v \in S$, 
and put $G_S(K)=\mathrm{Gal}(K_S/K)$. 
Since only pro-$p$-extensions over $k$ are treated here, 
we assume that the absolute norm $|N_{k/\mathbb Q}v| \equiv 1 \pmod{p}$ 
if $v \in S\setminus P$ is a prime ideal, 
and that $S$ contains no (real) archimedean primes if $p \neq 2$. 
The finitely presented pro-$p$ group $G_S(k)$ has a Koch type presentation in some basic cases 
(cf.\ e.g.\ \cite{KinH,Koch,Win}), 
and has been studied 
with a viewpoint of arithmetic topology 
(cf.\ \cite{Ama14,Mor02,Mor04,Vog05} etc.). 

Let $k^{\mathrm{cyc}}$ be the cyclotomic $\mathbb Z_p$-extension of $k$. 
A main object of Iwasawa theory is the $S$-ramified Iwasawa module $G_S(k^{\mathrm{cyc}})^{\mathrm{ab}}$, which is the abelianization of $G_S(k^{\mathrm{cyc}})$. 
If $P \subset S$, then $k^{\mathrm{cyc}} \subset k_S$, 
and hence $G_S(k^{\mathrm{cyc}})^{\mathrm{ab}}$ can be studied as a subquotient of $G_S(k)$ with the action of $\mathrm{Gal}(k^{\mathrm{cyc}}/k)$ 
induced by inner automorphism. 
On the other hand, we assume that $S \cap P=\emptyset$ throughout the following. 
In a similar way, 
the tamely ramified Iwasawa module $G_S(k^{\mathrm{cyc}})^{\mathrm{ab}}$ 
can be also studied as a subquotient of 
the Galois group 
\[
\widetilde{G}_S(k)=\mathrm{Gal}((k^{\mathrm{cyc}})_S/k) 
\simeq G_S(k^{\mathrm{cyc}}) \rtimes \mathrm{Gal}(k^{\mathrm{cyc}}/k) 
\]
of the maximal pro-$p$-extension $(k^{\mathrm{cyc}})_S/k$ which is unramified outside $S \cup P$ and `cyclotomically ramified' at any $v \in P$. 
Since the pro-$p$ group $\widetilde{G}_S(k)$ is also finitely presented (cf.\ \cite{BLM13,Sal08}), 
we can know the structure of any subquotient in principle if we obtain the explicit presentation of $\widetilde{G}_S(k)$ by generators and relations. 
In particular when $S=\emptyset$, 
$(k^{\mathrm{cyc}})_{\emptyset}$ is the union of the $p$-class field towers of 
intermediate fields of $k^{\mathrm{cyc}}/k$, 
and the closed subgroup $G_{\emptyset}(k^{\mathrm{cyc}})$ of $\widetilde{G}_{\emptyset}(k)$ is a central object in 
nonabelian Iwasawa theory of $\mathbb Z_p$-extensions (\cite{Oza07}). 
The purpose of this paper is to study these subjects from a viewpoint of 
analogy between Iwasawa theory and Alexander-Fox theory 
(cf.\ \cite{HMM,KM08,Maz,Mor,Uek} etc.), 
regarding the pro-$p$ Galois group $\widetilde{G}_S(k)$ as an analogue of a link group $\pi_1(\mathsf{X})$ (cf.\ Remark \ref{rem:anlgG}). 
A summary of results in each section is the following. 

In Section \ref{sec:genrelrk}, 
we recall Salle's Shafarevich type formula 
on the generator rank and the relation rank of $\widetilde{G}_S(k)$, 
by refering the results and arguments of \cite{BLM13,Sal08} 
with a little refinement in the case where $p=2$. 

In Section \ref{sec:pres}, 
we obtain a Koch type presentation of $\widetilde{G}_S(k)$ in some basic cases 
(Theorems \ref{thm:presQ} and \ref{thm:presk}), 
where the relations modulo the $3$rd step of the lower central series 
are described by linking numbers of primes. 
If $k=\mathbb Q$ and the archimedean prime $\infty \in S$, 
the R\'edei symbols appear in the relations modulo the $4$th step of the Zassenhaus filtration as the mod $2$ Milnor numbers (Proposition \ref{Redei:ciab} and Corollary \ref{cor:free4}). 
In particular, we obtain a triple of Borromean primes including $p=2$ (Example \ref{ex:Borromean}). 
Moreover, 
by the arguments as in \cite{BLM13} etc., 
$\widetilde{G}_S(\mathbb Q)$ is a mild pro-$p$ group of deficiency zero 
if $\infty \not\in S$ and $S \cup \{p\}$ forms a `circular set' of primes 
(cf.\ Section \ref{sbsec:mild}). 
Then the cohomological dimension $cd(\widetilde{G}_S(\mathbb Q))=2$. 

In Section \ref{sec:Iwapoly}, 
we focus on Iwasawa polynomials, 
which are defined as the characteristic polynomials of Iwasawa modules, 
and certainly analogous to Alexander polynomials (cf.\ Remark \ref{rem:anlgA}). 
Then 
the Koch type presentation of $\widetilde{G}_{\emptyset}(k)$ induces another proof of Gold's theorem \cite{Gol74} (Theorem \ref{thm:Gold}). 
If $\mathbb Q_S$ contains a quadratic extension $K/\mathbb Q$ ramified at any $v \in S\setminus\{\infty\}$, 
the unramified Iwasawa module $X=G_{\{\infty\}}(K^{\mathrm{cyc}})^{\mathrm{ab}}$ 
(in the narrow sense) 
is a subquotient of $\widetilde{G}_S(\mathbb Q)$. 
Following the arguments by Fr\"ohlich and Koch 
(cf.\ \cite{KinH,Koch,Vog05}), 
we obtain an approximation of the initial Fitting ideal of $X$ 
from the Koch type presentation of $\widetilde{G}_S(\mathbb Q)$ 
by the pro-$2$ Fox free differential calculus (Theorem \ref{thm:approx}). 
In particular, for a certain family of imaginary quadratic fields $K$, 
we calculate an approximation of the quadratic Iwasawa polynomial of $X$, 
which is described by $4$th power residue symbols (Theorem \ref{thm:d=2:imag}). 
An explicit presentation of 
$\widetilde{G}_{\{\infty\}}(K)=\mathrm{Gal}((K^{\mathrm{cyc}})_{\{\infty\}}/K)$ 
is also calculated for a certain family of real quadratic fields $K$ 
(Theorem \ref{thm:d=2:real}). 

\begin{notations}
For a pro-$p$ group $G$, 
we denote by $[h,g]=h^{-1}g^{-1}hg$ the commutator of $g,h \in G$. 
For a closed subgroup $H$ of $G$, 
$[H,G]$ (resp.\ $H^{p^n}$) denotes the minimal closed subgroup containing $\{[h,g]\,|\,h \in H,g \in G\}$ (resp.\ $\{h^{p^n}\,|\,h \in H\}$). 
Then $G^{\mathrm{ab}}=G/[G,G]$. 
The $i$th cohomology group with coefficients in $\mathbb F_p=\mathbb Z/p\mathbb Z$ is denoted by $H^i(G)=H^i(G,\mathbb Z/p\mathbb Z)$. 
For a set $Y$, $|Y|$ denotes the cardinality. 
For objects $x$ and $y$, $\delta_{x,y}=1$ if $x=y$, and $\delta_{x,y}=0$ otherwise. 
\end{notations}

%------------------------------------------------------------------------------
\section{Generator rank and relation rank}\label{sec:genrelrk}

\subsection{Pro-$p$ class field theory.} 
First, we recall the id\`elic class field theory by \cite{Jau98}. 
For each prime $v$ of $k$, 
we put $\widehat{k_v^{\times}}= \varprojlim k_v^{\times}/(k_v^{\times})^{p^n}$, 
where $k_v$ is the completion of $k$ at $v$. 
For a nonarchimedean prime $v$, we put $\mathcal{U}_v = \varprojlim U_v/U_v^{p^n} \subset \widehat{k_v^{\times}}$, 
where $U_v = \mathrm{Ker}(k_v^{\times} \stackrel{v}{\rightarrow} \mathbb Z)$ is the unit group of the local field $k_v$. 
The local class field theory yields that 
$\mathrm{Gal}(k_v^{\mathrm{ab},p}/k_v) \simeq \widehat{k_v^{\times}}$ 
and $\mathrm{Gal}(k_v^{\mathrm{ur},p}/k_v) \simeq \widehat{k_v^{\times}}/\mathcal{U}_v$, 
where $k_v^{\mathrm{ab},p}$ denotes the maximal abelian pro-$p$-extension of $k_v$, and $k_v^{\mathrm{ur},p}$ denotes the unramified $\mathbb Z_p$-extension of $k_v$. 
For the cyclotomic $\mathbb Z_p$-extension $k_v^{\mathrm{cyc}}$ of $k_v$,
put 
\[
\mathcal{N}_v=\mathrm{Ker}(\widehat{k_v^{\times}} \twoheadrightarrow \varprojlim k_v^{\times}/N_{k_{v,n}^{\mathrm{cyc}}/k_v} (k_{v,n}^{\mathrm{cyc}})^{\times}) \simeq \mathrm{Ker}(\mathrm{Gal}(k_v^{\mathrm{ab},p}/k_v) \twoheadrightarrow \mathrm{Gal}(k_v^{\mathrm{cyc}}/k_v)) , 
\]
where $k_{v,n}^{\mathrm{cyc}}$ denotes the subextension of $k_v^{\mathrm{cyc}}$ 
such that $[k_{v,n}^{\mathrm{cyc}}:k_v]=p^n$. 
Then we put $\widetilde{\mathcal{U}}_v=\mathcal{U}_v \cap \mathcal{N}_v$. 
Note that 
$\widetilde{\mathcal{U}}_v=\mathcal{U}_v=\mathcal{N}_v$ (i.e., $k_v^{\mathrm{cyc}}=k_v^{\mathrm{ur},p}$) 
if $v \not\in P$, 
and that 
$\widehat{k_v^{\times}}/\widetilde{\mathcal{U}}_v \simeq \mathrm{Gal}(k_v^{\mathrm{ur},p}k_v^{\mathrm{cyc}}/k_v) \simeq \mathbb Z_p^2$ 
if $v \in P$. 
On the other hand, 
we put $\widetilde{\mathcal{U}}_v=\mathcal{U}_v=\widehat{k_v^{\times}}$ for archimedean $v|\infty$, 
where we note that $\widehat{k_v^{\times}} \simeq \mathbb R^{\times}/\mathbb R_{>0}$ if $p=2$ and $v$ is real, and $\widehat{k_v^{\times}}=1$ otherwise. 

\begin{example}[cf.\ e.g.\ {\cite[Appendix \S3]{Was}}]\label{exam:kvQp}
Suppose that $k_v=\mathbb Q_p$. 
Then $\widehat{k_v^{\times}}=p^{\mathbb Z_p}\mathcal{U}_v$. 
If $p \neq 2$, 
we have $\mathcal{U}_v=(1+p)^{\mathbb Z_p}$, $\mathcal{N}_v=p^{\mathbb Z_p}$ and $\widetilde{\mathcal{U}}_v=\{1\}$. 
If $p=2$, 
we have $\mathcal{U}_v=\langle -1 \rangle 5^{\mathbb Z_2}$, $\mathcal{N}_v=\langle -1 \rangle 2^{\mathbb Z_2}$ and $\widetilde{\mathcal{U}}_v=\langle -1 \rangle$. 
\end{example}

Let 
$\mathcal{J}_k= \mathrm{Ker}(
\prod_v \widehat{k_v^{\times}} \rightarrow 
\prod_v (\widehat{k_v^{\times}}/\mathcal{U}_v)
/\bigoplus_v (\widehat{k_v^{\times}}/\mathcal{U}_v)
)$ 
be the id\`eles of $k$ as defined in \cite{Jau98}, 
and put $\overline{k^{\times}}=k^{\times} \otimes \mathbb Z_p \subset \mathcal{J}_k$, 
where we identify $\overline{k^{\times}}$ with the image of the diagonal embedding $\overline{k^{\times}} \rightarrow \varprojlim k^{\times}/(k^{\times})^{p^n} \stackrel{\mathrm{diag.}}{\longrightarrow} \mathcal{J}_k$ (cf.\ \cite[Remaques 1 (iii), Th\'eor\`eme et d\'efinition 1.4]{Jau98}). 
We also identify $\widehat{k_w^{\times}}$ with $\widehat{k_w^{\times}} \times \prod_{v \neq w}\{1\} \subset \mathcal{J}_k$ for each prime $w$ of $k$. 
The reciprocity map in the pro-$p$ version of global class field theory \cite{Jau98} 
is the isomorphism 
\begin{align}\label{idele:rec}
\mathrm{rec} : \mathcal{J}_k/\overline{k^{\times}} \stackrel{\simeq}{\longrightarrow} \mathrm{Gal}(k^{\mathrm{ab},p}/k)
\end{align}
such that $\mathrm{rec}(\mathcal{U}_v \overline{k^{\times}}/\overline{k^{\times}})$ (resp.\ $\mathrm{rec}(\widehat{k_v^{\times}} \overline{k^{\times}}/\overline{k^{\times}})$) is the inertia group (resp.\ decomposition group) of each prime $v$, 
where $k^{\mathrm{ab},p}$ is the maximal abelian pro-$p$-extension of $k$. 
In particular, we have 
$\mathcal{J}_k/\mathcal{U}\overline{k^{\times}} \simeq G_{\emptyset}(k)^{\mathrm{ab}}$, 
where $\mathcal{U}= \prod_v \mathcal{U}_v$. 
Recall that $S \cap P =\emptyset$. 
Put 
\[
\widetilde{\mathcal{U}}_S 
= 
\underset{v \not\in S}{\textstyle{\prod}} \widetilde{\mathcal{U}}_v 
= 
\underset{v \not\in S \cup P}{\textstyle{\prod}} \mathcal{U}_v 
\times 
\underset{v \in P}{\textstyle{\prod}} \widetilde{\mathcal{U}}_v 
\times \underset{v \in S}{\textstyle{\prod}} \{1\} 
\subset \mathcal{U} . 
\]
Then we have the following isomorphism. 

\begin{proposition}[cf.\ {\cite[Proposition 1.2]{Sal08}}]\label{prop:idele}
The reciprocity map (\ref{idele:rec}) induces an isomorphism 
\[
\mathrm{rec}_S : \mathcal{J}_k/\widetilde{\mathcal{U}}_S\overline{k^{\times}} \stackrel{\simeq}{\longrightarrow} \widetilde{G}_S(k)^{\mathrm{ab}} . 
\]
\end{proposition}

\begin{proof}
See the proof of \cite[Proposition 1.2]{Sal08} 
with the consideration of the ramification of $v|\infty$. 
The same arguments hold for $S$ containing some $v|\infty$. 
\end{proof}

\subsection{Generator rank and Kummer groups.} 
The map 
$\overline{k^{\times}}=k^{\times} \otimes \mathbb Z_p \rightarrow \varprojlim k^{\times}/(k^{\times})^{p^n} \twoheadrightarrow k^{\times}/(k^{\times})^p$ 
induces an isomorphism 
\begin{align}\label{kkp_kkp}
k^{\times}/(k^{\times})^p \simeq \overline{k^{\times}}/(\overline{k^{\times}})^p : a(k^{\times})^p \mapsto (a \otimes 1)(\overline{k^{\times}})^p . 
\end{align}
Put $V=\mathcal{U}\mathcal{J}_k^p \cap \overline{k^{\times}}$ 
and 
$\widetilde{V}_S = \widetilde{\mathcal{U}}_S \mathcal{J}_k^p \cap \overline{k^{\times}}$. 
Then $\overline{k^{\times}}^p \subset \widetilde{V}_S \subset V$. 
Since 
$\widehat{k_v^{\times}}/(\widehat{k_v^{\times}})^p \simeq k_v^{\times}/(k_v^{\times})^p$ and $\mathcal{U}_v/\mathcal{U}_v^p \simeq U_v/U_v^p$ for each $v$, 
and since 
$\widehat{k_v^{\times}}/\mathcal{N}_v(\widehat{k_v^{\times}})^p \simeq k_v^{\times}/N_{k^{\mathrm{cyc}}_{v,1}/k_v}(k^{\mathrm{cyc}}_{v,1})^{\times}$ 
for $v \in P$, 
we have 
\begin{align}\label{V_B}
\widetilde{V}_S/(\overline{k^{\times}})^p 
\simeq \widetilde{B}_S/(k^{\times})^p 
\subset B_S/(k^{\times})^p 
\subset B_{\emptyset}/(k^{\times})^p
\simeq V/(\overline{k^{\times}})^p
\end{align}
under the isomorphism (\ref{kkp_kkp}), 
where $B_{\emptyset}$ consists of $a \in k^{\times}$ 
such that the principal ideal $(a)=\mathfrak{a}^p$ is a $p$th power of an ideal $\mathfrak{a}$ of $k$, 
and 
$B_S=\{a \in B_{\emptyset}\,|\,a \in (k_v^{\times})^p \hbox{ for all $v \in S$} \}$, 
\[
\widetilde{B}_S=\{a \in B_S\,|\,a \in (N_{k^{\mathrm{cyc}}_{v,1}/k_v}(k^{\mathrm{cyc}}_{v,1})^{\times} \cap U_v)(k_v^{\times})^p \hbox{ for all $v \in P$} \} . 
\]
There is also an exact sequence 
\begin{align}\label{E_B_A}
\entrymodifiers={+!!<0pt,\fontdimen22\textfont2>}
\xymatrix{
0 \ar[r] & E_k/E_k^p \ar[r] & B_{\emptyset}/(k^{\times})^p \ar[rr]^-{a(k^{\times})^p \mapsto [\mathfrak{a}]} && Cl(k) \ar[r]^-{p} & Cl(k) 
}, 
\end{align}
where $E_k$ is the unit group of $k$, and $Cl(k)$ is the ideal class group of $k$. 

\begin{remark}\label{rem:trivB}
The following facts are induced from (\ref{V_B}) and (\ref{E_B_A}): 
If $p \neq 2$ and either $k=\mathbb Q$ or $k$ is an imaginary quadratic field with class number $h_k=|Cl(k)|$ not divisible by $p$, 
then $\widetilde{B}_S=B_{\emptyset}=(k^{\times})^p$ 
except when $p=3$ and $k=\mathbb Q(\sqrt{-3})$. 
If $p=2$ and $k=\mathbb Q$, and if $S$ contains $\infty$ or a prime number $q \equiv 3 \pmod{4}$, then $\widetilde{B}_S=B_S=(k^{\times})^p$ (cf.\ \cite[Example 11.12]{Koch}). 
\end{remark}

Then we have the following formula. 

\begin{proposition}[cf.\ {\cite[Th\'eor\`eme 3.3]{Sal08}}]\label{prop:genrk}
Under the settings above, 
the minimal number $\widetilde{d}_S$ of generators of $\widetilde{G}_S(k)$ is 
\[
\widetilde{d}_S=\dim_{\mathbb F_p}H^1(\widetilde{G}_S(k))=
|S|+|P|+
\dim_{\mathbb F_p} \widetilde{B}_S/(k^{\times})^p - \dim_{\mathbb F_p} E_k/E_k^p . 
\]
\end{proposition}

\begin{proof}
See the proof of \cite[Th\'eor\`eme 3.3]{Sal08} (cf.\ also \cite[Theorem 3.5 (i)]{BLM13}) with the consideration of the ramification of $v|\infty$. 
Since $\mathcal{U}_v/\mathcal{U}_v^p \simeq \mathbb Z/p\mathbb Z$ 
also for archimedean $v \in S$, the same arguments hold. 
\end{proof}

\subsection{Relation rank and the Shafarevich kernel.}\label{sbsec:relrk} 
Let $k_v^{\mathrm{pro}\hbox{\scriptsize -}p}$ (resp.\ $k^{\mathrm{pro}\hbox{\scriptsize -}p}$) be the maximal pro-$p$-extension of $k_v$ (resp.\ $k$). 
We fix an embedding $k^{\mathrm{pro}\hbox{\scriptsize -}p} \subset k_v^{\mathrm{pro}\hbox{\scriptsize -}p}$ 
for each $v$. 
Then 
\[
\overline{G}_v=\mathrm{Gal}(k_v^{\mathrm{pro}\hbox{\scriptsize -}p}/k_v) 
\]
is identified with the decomposition subgroup of 
$\overline{G}=\mathrm{Gal}(k^{\mathrm{pro}\hbox{\scriptsize -}p}/k)$ 
for a fixed prime $\overline{v}|v$ of $k^{\mathrm{pro}\hbox{\scriptsize -}p}$ 
(cf.\ e.g.\ \cite[II, \S 6.1]{Serre}), 
and 
$\widetilde{G}_{S,v}=\mathrm{Gal}((k^{\mathrm{cyc}})_Sk_v/k_v)$ 
is also identified with the decomposition subgroup of 
$\widetilde{G}_S(k)=\mathrm{Gal}((k^{\mathrm{cyc}})_S/k)$ 
for the prime $\widetilde{v}$ of $(k^{\mathrm{cyc}})_S$ such that $\overline{v}|\widetilde{v}$. 
Since the maximal unramified pro-$p$-extension of $k_v^{\mathrm{cyc}}$ is $k_v^{\mathrm{ur},p}k_v^{\mathrm{cyc}}$, 
we have $(k^{\mathrm{cyc}})_Sk_v \subset k_v^{\mathrm{ur},p}k_v^{\mathrm{cyc}}$ 
if $v \not\in S$ and $v \nmid \infty$, particularly if $v \in P$. 
For each $v \in P$, we denote by 
\[
G_v^{\mathrm{cr}}=\mathrm{Gal}(k_v^{\mathrm{ur},p}k_v^{\mathrm{cyc}}/k_v) \simeq \mathbb Z_p^2 
\]
the maximal `cyclotomically ramified' quotient of $\overline{G}_v$. 
Put 
\[
\Sha^2(\widetilde{G}_S(k))=\mathrm{Ker}\big(
H^2(\widetilde{G}_S(k)) 
\stackrel{\mathrm{loc}}{\longrightarrow} 
\underset{v \in S^{\ast}}{\textstyle{\bigoplus}} H^2(\overline{G}_v) 
\oplus 
\underset{v \in P}{\textstyle{\bigoplus}} H^2(G_v^{\mathrm{cr}})
\big)  
\]
with the localization map $\mathrm{loc}$ 
induced from 
$\{\overline{G}_v \stackrel{\varphi_v}{\twoheadrightarrow} \widetilde{G}_{S,v} \subset \widetilde{G}_S(k)\}_{v \in S^{\ast}}$ 
and 
$\{G_v^{\mathrm{cr}} \stackrel{\varphi_v}{\twoheadrightarrow} \widetilde{G}_{S,v} \subset \widetilde{G}_S(k)\}_{v \in P}$, 
where $S^{\ast}=S \setminus \{w\}$ with an arbitrary (but suitable) $w \in S$ if $S \neq \emptyset$ and $k$ contains a primitive $p$th root $\zeta_p$ of unity, and $S^{\ast}=S$ otherwise. 

By \cite[Th\'eor\`eme 4.1]{Sal08} with a refinement in the case where $p=2$, we obtain the following theorem. 

\begin{theorem}[cf.\ {\cite[Th\'eor\`eme 4.1]{Sal08}}]\label{thm:relrk}
Under the settings above, an inequality 
\begin{align}\label{relrk:Sha<B}
\dim_{\mathbb F_p} \Sha^2(\widetilde{G}_S(k)) \le 
\dim_{\mathbb F_p} \widetilde{B}_S/(k^{\times})^p 
\end{align}
is satisfied, and the minimal number $\widetilde{r}_S$ of relations of $\widetilde{G}_S(k)$ satisfies 
\begin{align}\label{relrk:relrk}
\widetilde{r}_S = \dim_{\mathbb F_p}H^2(\widetilde{G}_S(k)) \le 
\dim_{\mathbb F_p} \Sha^2(\widetilde{G}_S(k)) 
+|S|-(1-\delta_{S,\emptyset})\theta +|P| , 
\end{align}
where $\theta=1$ if $\zeta_p \in k$, and $\theta=0$ otherwise. 
\end{theorem}

\begin{proof}
Following the proof of \cite[Theorem 3.5]{BLM13} (cf.\ also \cite[Th\'eor\`eme 4.1]{Sal08}), 
we give a summary proof valid also for $p=2$. 
Let $\overline{I}_v$ be the inertia subgroup of $\overline{G}_v$. 
Since $\overline{G}_v$ is the decomposition subgroup of $\overline{G}$ for $\overline{v}|v$, 
and since $k_v^{\mathrm{cyc}}=k^{\mathrm{cyc}}k_v$, 
$\overline{I}_v \cap \mathrm{Gal}(k_v^{\mathrm{pro}\hbox{\scriptsize -}p}/k_v^{\mathrm{cyc}})$ 
is the inertia subgroup of 
$\mathrm{Gal}(k^{\mathrm{pro}\hbox{\scriptsize -}p}/k^{\mathrm{cyc}})$ 
for $\overline{v}$. 
For each $v$, we put $H_v = \overline{G}_v/D_v$ with $D_v$ defined as follows: 
\begin{itemize}

\item[$\cdot$]
$D_v=\overline{I}_v$ if $v \not\in S \cup P$. 
Then $H_v \simeq \mathbb Z_p$ if $v \nmid \infty$, and $H_v \simeq 1$ if $v|\infty$.

\item[$\cdot$]
$D_v = \overline{I}_v \cap \mathrm{Gal}(k_v^{\mathrm{pro}\hbox{\scriptsize -}p}/k_v^{\mathrm{cyc}})$ if $v \in P$. 
Then $H_v=G_v^{\mathrm{cr}} \simeq \mathbb Z_p^2$. 

\item[$\cdot$]
$D_v=\{1\}$ if $v \in S$. 
Then $H_v=\overline{G}_v$. 

\end{itemize}
Then $\widetilde{G}_S=\widetilde{G}_S(k) \simeq \overline{G}/D_S$, 
where 
$D_S=\langle \bigcup_{v \not\in S} D_v \rangle_{\overline{G}}$ 
is the minimal closed normal subgroup of $\overline{G}$ containing $D_v$ for all $v \not\in S$. 
Put 
\[
P_0=\{v \in P\,|\, 0 \rightarrow H^2(H_v) \stackrel{\mathrm{inf}}{\longrightarrow} H^2(\overline{G}_v) \hbox{\ (exact)} \} 
\]
(For example, $P_0=\emptyset$ if $p \neq 2$, and $P_0=P$ if $p=2$ and $k=\mathbb Q$. 
cf.\ \cite[Lemma 2.2]{BLM13} and \cite[Proposition 4.8 (4)]{Sal08}.) 
As in \cite[\S 4]{Sal08}, 
the five term exact sequence induces the exact rows of 
commutative diagrams 
\begin{align}\label{diag:5term}
\entrymodifiers={+!!<0pt,\fontdimen22\textfont2>}
\xymatrix@R=0.5\baselineskip{
0 \ar[r] & H^1(\overline{G})/H^1(\widetilde{G}_S) \ar^-{\mathrm{res}}[r] \ar^-{\phi}[d] 
& H^1(D_S)^{\widetilde{G}_S} \ar^-{\mathrm{tg}}[r] \ar^-{\phi_1}[d] & \mathrm{Ker}(\mathrm{Inf}_1) \ar[r] \ar^-{\phi_2}[d] & 0\\
0 \ar[r] & \underset{v \not\in S}{\bigoplus} H^1(\overline{G}_v)/H^1(H_v) \ar^-{\mathrm{res}}[r] & \underset{v \not\in S}{\bigoplus} H^1(D_v)^{H_v} \ar^-{\mathrm{tg}}[r] & \mathrm{Ker}(\mathrm{Inf}_2) \ar[r] & 0 
}
\end{align}
and 
\begin{align}\label{diag:Sha}
\entrymodifiers={+!!<0pt,\fontdimen22\textfont2>}
\xymatrix@R=0.6\baselineskip{
0 \ar[r] & \mathrm{Ker}(\mathrm{Inf}_1) \ar[r] \ar[d] 
\ar@/_{4.5pc}/[dd]_>>>>>>>>>{\phi_2}|(.47)\hole 
& H^2(\widetilde{G}_S) \ar^-{\mathrm{Inf}_1}[r] \ar^-{\mathrm{loc}'}[d] & H^2(\overline{G}) \ar^-{\mathrm{Res}}[d] \\
0 \ar[r] & \underset{v \in P \setminus P_0}{\bigoplus} H^2(H_v) \ar[r] \ar^-{\simeq}[d] & \underset{v \not\in S \setminus S^*}{\bigoplus} H^2(H_v) \ar^-{\mathrm{inf}}[r] \ar[d] & \underset{v \not\in S \setminus S^*}{\bigoplus} H^2(\overline{G}_v) \ar[d] \\
0 \ar[r] & \mathrm{Ker}(\mathrm{Inf}_2) \ar[r] & \underset{v \not\in S}{\bigoplus} H^2(H_v) \ar^-{\mathrm{Inf}_2}[r] & \underset{v \not\in S}{\bigoplus} H^2(\overline{G}_v) 
}
\end{align}
where we note the following: 
The injective map $\phi_1$ is well-defined 
as the dual of $\phi_1^{\vee} : \textstyle{\prod}_{v \not\in S} D_v/D_v^p[D_v,\overline{G}_v] \twoheadrightarrow D_S/D_S^p[D_S,\overline{G}]$, 
and hence $\phi$ is also well-defined. 
Since $H^2(H_v) \simeq 0$ if $v \not\in S \cup P$, 
the map $\mathrm{loc}'$ is also well-defined. 
Since $\mathrm{Res}$ is well-defined and injective (cf.\ \cite[Theorems 11.1 and 11.2]{Koch}), 
we have 
$\Sha^2(\widetilde{G}_S(k))=\mathrm{Ker}(\mathrm{loc}') \simeq \mathrm{Ker}\,\phi_2$ 
by the snake lemma for a diagram induced from the upper half of (\ref{diag:Sha}). 
The snake lemma for (\ref{diag:5term}) yields an injective homomorphism 
\[
\Sha^2(\widetilde{G}_S(k)) \hookrightarrow \mathrm{Coker}\,\phi . 
\]
There is also a commutative diagram 
\[
\entrymodifiers={+!!<0pt,\fontdimen22\textfont2>}
\xymatrix@R=0.9\baselineskip{
& \widetilde{\mathcal{U}}_S/\widetilde{\mathcal{U}}_S^p \ar^-{\xi}[r] \ar^-{\xi'}[d] & \mathcal{J}_k/\overline{k^{\times}}\mathcal{J}_k^p  \ar[r] \ar^-{\simeq}[d] & \mathcal{J}_k/\widetilde{\mathcal{U}}_S \overline{k^{\times}}\mathcal{J}_k^p \ar[r] \ar^-{\simeq}[d] & 0 \\
0 \ar[r] & D_S/(D_S \cap \overline{G}^p[\overline{G},\overline{G}]) \ar[r] & \overline{G}^{\mathrm{ab}}/p  \ar[r] & \widetilde{G}_S^{\mathrm{ab}}/p \ar[r] & 0 \\
0 \ar[r] & \underset{v \not\in S}{\textstyle{\prod}} D_v/D_v^p[\overline{G}_v,\overline{G}_v] \ar[r] \ar^-{\phi^{\vee}}[u] \ar@/^{6pc}/[uu]^>>>>>{\simeq}|(.53)\hole & \underset{v \not\in S}{\textstyle{\prod}} \overline{G}_v^{\mathrm{ab}}/p \ar[r] & \underset{v \not\in S}{\textstyle{\prod}} H_v/p \ar[r] & 0 
}
\]
with exact rows, 
where we note that $D_v \cap \overline{G}_v^p[\overline{G}_v,\overline{G}_v] 
=D_v^p[\overline{G}_v,\overline{G}_v]$ if $v \not\in S$ by \cite[Lemma 3.9]{BLM13}. 
Since the continuous surjective homomorphism $\phi^{\vee}$ is the dual of $\phi$ in (\ref{diag:5term}), 
and since $\overline{k^{\times}} \cap \mathcal{J}_k^p = (\overline{k^{\times}})^p$ 
by Hasse principle (cf.\ e.g.\ \cite[II, Theorem 6.3.3]{Gra}) and (\ref{kkp_kkp}), 
we have 
\begin{align*}
\mathrm{Hom}(\mathrm{Coker}\,\phi,\mathbb F_p) 
&\simeq \mathrm{Ker}\,\phi^{\vee} \simeq \mathrm{Ker}\,\xi' \simeq \mathrm{Ker}\,\xi 
\\
&\simeq 
(\widetilde{\mathcal{U}}_S \cap \overline{k^{\times}}\mathcal{J}_k^p)\mathcal{J}_k^p/\mathcal{J}_k^p
=
\widetilde{V}_S\mathcal{J}_k^p/\mathcal{J}_k^p
\simeq 
\widetilde{V}_S/(\overline{k^{\times}})^p . 
\end{align*}
Therefore the injection 
$\Sha^2(\widetilde{G}_S(k)) \hookrightarrow \mathrm{Hom}(\widetilde{B}_S/(k^{\times})^p,\mathbb F_p)$ 
yields the inequality (\ref{relrk:Sha<B}).  
If $v \in S \cup P$ and $v \nmid \infty$, then $H_v \simeq \mathbb Z_p \rtimes \mathbb Z_p$ (cf.\ \cite[Theorem 10.2]{Koch}). 
Since $H^2(H_v) \simeq \mathbb Z/p\mathbb Z$ if $v \in S^* \cup P$, 
the inequality (\ref{relrk:relrk}) also holds by the definition of $\Sha^2(\widetilde{G}_S(k))$. 
\end{proof}

Several consequences are also refined when $p=2$ as follows. 

\begin{corollary}[cf.\ {\cite[Corollarie 4.3]{Sal08}}]
The second partial Euler-Poincar\'e characteristic $\chi_2(\widetilde{G}_S(k))=1-\widetilde{d}_S+\widetilde{r}_S$ satisfies 
\[
\chi_2(\widetilde{G}_S(k)) \le 1+\dim_{\mathbb F_p} E_k/E_k^p - (1-\delta_{S,\emptyset})\theta . 
\]
In particular, $\chi_2(\widetilde{G}_S(\mathbb Q)) \le 1$. 
\end{corollary}

\begin{proof}
The inequalities are induced from 
Proposition \ref{prop:genrk} and Theorem \ref{thm:relrk}, 
where we note that 
$\widetilde{G}_{\emptyset}(\mathbb Q) \simeq G_{\{p\}}(\mathbb Q) \simeq \mathbb Z_p$. 
\end{proof}

\begin{corollary}[cf.\ {\cite[Corollary 3.11]{BLM13}}]
If 
\[
|S|+|P| \ge \dim_{\mathbb F_p}E_k/E_k^p+2+2\sqrt{1+\dim_{\mathbb F_p}E_k/E_k^p -(1-\delta_{S,\emptyset})\theta} , 
\]
then $\widetilde{G}_S(k)$ and $G_S(k^{\mathrm{cyc}})$ are not $p$-adic analytic. 
In particular, 
$\widetilde{G}_S(\mathbb Q)$ and $G_S(\mathbb Q^{\mathrm{cyc}})$ are not $p$-adic analytic if $|S| \ge 3+\delta_{p,2}$. 
Moreover, 
$\widetilde{G}_{\emptyset}(k)$ and $G_{\emptyset}(k^{\mathrm{cyc}})$ are not $p$-adic analytic if 
$k/\mathbb Q$ is a totally imaginary extension of degree 
$[k:\mathbb Q] \ge 4(3+\theta)$ in which $p$ splits completely. 
\end{corollary}

\begin{proof}
Note that $\dim_{\mathbb F_p}E_k/E_k^p \ge 1$ if $\theta=1$. 
The inequality $1+\dim_{\mathbb F_p}E_k/E_k^p -(1-\delta_{S,\emptyset})\theta \le 0$ can not occur. 
Hence $\widetilde{d}_S \ge |S|+|P|-\dim_{\mathbb F_p}E_k/E_k^p \ge 3$ 
by Proposition \ref{prop:genrk} and the assumption. 
Suppose that $\widetilde{G}_S(k)$ is $p$-adic analytic. 
Then $\widetilde{G}_S(k)$ has `finite rank' (cf.\ \cite[Theorem 8.36]{DDMS}), 
and hence the Golod-Shafarevich inequality 
$\widetilde{r}_S > \frac{1}{4}(\widetilde{d}_S)^2$ 
is satisfied (cf.\ \cite[Theorem D1]{DDMS}). 
By Proposition \ref{prop:genrk} and Theorem \ref{thm:relrk}, 
we have 
\[
(\widetilde{d}_S)^2 < 4\widetilde{r}_S \le 4(\widetilde{d}_S+\dim_{\mathbb F_p}E_k/E_k^p -(1-\delta_{S,\emptyset})\theta) , 
\]
and hence 
\[
|S|+|P|-\dim_{\mathbb F_p}E_k/E_k^p \le \widetilde{d}_S < 2+2\sqrt{1+\dim_{\mathbb F_p}E_k/E_k^p -(1-\delta_{S,\emptyset})\theta} . 
\]
This contradicts the assumption. 
Therefore $\widetilde{G}_S(k)$ is not $p$-adic analytic. 
Since $\widetilde{G}_S(k)/G_S(k^{\mathrm{cyc}}) \simeq \mathbb Z_p$, 
$G_S(k^{\mathrm{cyc}})$ is also not $p$-adic analytic. 
(cf.\ \cite[Exercise 3.1, Corollary 8.33, Theorem 8.36]{DDMS}). 
\end{proof}

\begin{remark}
All sets $S$ with prometacyclic $G_S(\mathbb Q^{\mathrm{cyc}})$ 
have been characterized arithmetically 
(cf.\ \cite{IM14,Miz**,MO13}). 
Then $\widetilde{G}_S(\mathbb Q)$ is $p$-adic analytic, and actually there exists such $S$ of cardinality $|S| = 2+\delta_{p,2}$. 
There are other examples of prometacyclic $G_S(k^{\mathrm{cyc}})$ 
for $p=2$ and quadratic $k/\mathbb Q$ 
(cf.\ \cite{Miz05,Miz10a,Sal10} and Theorem \ref{thm:d=2:real}). 
Moreover, all imaginary quadratic fields $k$ with prometacyclic (or abelian) $G_{\emptyset}(k^{\mathrm{cyc}})$ have been characterized 
(cf.\ \cite{Miz10b,MO10,Oka06}). 
There also exists $p$ such that $G_{\emptyset}(\mathbb Q(\zeta_p)^{\mathrm{cyc}})$ is abelian and not procyclic (cf.\ \cite{Sha08}). 
\end{remark}

\begin{example}[cf.\ {\cite[Theorems 7.3 and 9.2]{Miz**}}]
Suppose that $p=2$ and $k=\mathbb Q$. 
If $S=\{\ell_1,\ell_2\}$ ($\ell_1 \neq \ell_2$) with $\ell_1 \equiv \ell_2 \equiv 3 \pmod{8}$, 
then $\widetilde{G}_S(\mathbb Q) \simeq \mathbb Z_2 \rtimes \mathbb Z_2$, 
$\widetilde{d}_S=2$, $\widetilde{r}_S=1$, and $\chi_2(\widetilde{G}_S(\mathbb Q))=0$. 
If $S=\{\ell,\infty\}$ with $\ell \equiv 7 \pmod{16}$, 
then $\widetilde{G}_S(\mathbb Q) \simeq D_{2^{\infty}} \rtimes \mathbb Z_2$, 
$\widetilde{d}_S=\widetilde{r}_S=2$, and $\chi_2(\widetilde{G}_S(\mathbb Q))=1$, 
where $G_S(\mathbb Q^{\mathrm{cyc}}) \simeq D_{2^{\infty}} \simeq \mathbb Z_2 \rtimes (\mathbb Z/2\mathbb Z)$ is a prodihedral pro-$2$ group. 
\end{example}

\begin{example}[cf.\ {\cite[Theorem 1]{MO13}}]
Suppose that $p \neq 2$, $k=\mathbb Q$ and $S=\{\ell_1,\ell_2\}$ 
with $\ell_1 \not\equiv 1 \pmod{p^2}$, $\ell_2 \not\equiv 1 \pmod{p^2}$ 
such that 
$G_{\emptyset}(K^{\mathrm{cyc}}) \simeq \{1\}$ for the $p$-extension $K=\mathbb Q_{\{\ell_1\}}\mathbb Q_{\{\ell_2\}}$ of degree $p^2$. 
(All such $S$ have been characterized arithmetically in \cite{Gen00}.) 
Then $\widetilde{G}_S(\mathbb Q) \simeq ((\mathbb Z/p^2\mathbb Z) \rtimes \mathbb Z_p) \rtimes \mathbb Z_p$, 
$\widetilde{d}_S=\widetilde{r}_S=3$, and $\chi_2(\widetilde{G}_S(\mathbb Q))=1$. 
Moreover, there is a minimal presentation 
\[
\entrymodifiers={+!!<0pt,\fontdimen22\textfont2>}
\xymatrix{
1 \ar[r] & R \ar[r] & F \ar[r] & \widetilde{G}_S(\mathbb Q) \ar[r] & 1
}
\]
with a free pro-$p$ group $F$ generated by $\{a,b,c\}$, 
and the normal subgroup $R$ normally generated by 
$\{ a^{-p}[a,b], [a,c], a^{-p}b^{-p}[b,c] \}$. 
One can see that $a^{p^2} \in R$ from a calculation of $c^{-1}(a^{-p}[a,b])c \in R$. 
There is also a similar example for $p=2$ (cf.\ \cite[Theorem 2]{MO13}). 
\end{example}

%------------------------------------------------------------------------------
\section{Finite presentation}\label{sec:pres}

\subsection{Koch type presentation.} 

Suppose that $k=\mathbb Q$. 
For each prime number $\ell \in S$, we fix an integer $\alpha_{\ell}$ 
such that $\alpha_{\ell} +\ell\mathbb Z$ is a generator of cyclic group $(\mathbb Z/\ell\mathbb Z)^{\times}$. 
%We fix a primitive root $\alpha_{\ell} \in \mathbb Z$ of each prime number $\ell \in S$. 
For a pair $(\ell,\ell')$ of prime numbers $\ell$, $\ell' \in S \cup \{p\}$, 
we define the linking number $\mathrm{lk}(\ell,\ell')$ as follows: 
If $\ell \neq \ell' \equiv 1 \pmod{p}$, 
then $\mathrm{lk}(\ell,\ell')$ is an integer such that 
\[
\ell^{-1} \equiv \alpha_{\ell'}^{\mathrm{lk}(\ell,\ell')} \mod{\ell'}
\]
and $0 \le \mathrm{lk}(\ell,\ell') < \ell'$. 
If $\ell \neq \ell'=p$, 
then $\mathrm{lk}(\ell,p) \in \mathbb Z_p$ is a $p$-adic integer satisfying 
\[
\ell=
\left\{
\begin{array}{ll}
(1+p)^{\mathrm{lk}(\ell,p)} & \hbox{if } p \neq 2, \\
(-1)^{\frac{\ell -1}{2}} 5^{\mathrm{lk}(\ell,p)} & \hbox{if } p=2. 
\end{array}
\right.
\]
If $\ell=\ell'$, we put $\mathrm{lk}(\ell,\ell')=0$. 

\begin{theorem}\label{thm:presQ}
Assume that $k=\mathbb Q$ and 
\[
S=\left\{
\begin{array}{ll}
\{\ell_1,\cdots,\ell_d\} & \hbox{if } p \neq 2, \\ 
\{\ell_1,\cdots,\ell_d,\infty\} \hbox{ or } \{\ell_1,\cdots,\ell_d,q\} & \hbox{if } p=2, 
\end{array}
\right.
\]
where $\ell_i \equiv 1 \pmod{p}$ for $1 \le i \le d$, and $q \equiv 3 \pmod{4}$. 
Put $\ell_0=p$. 
Then $\widetilde{G}_S(\mathbb Q)$ has a minimal presentation
\[
\entrymodifiers={+!!<0pt,\fontdimen22\textfont2>}
\xymatrix{
1 \ar[r] & R \ar[r] & F \ar[r]^-{\pi} & \widetilde{G}_S(\mathbb Q) \ar[r] & 1
}
\]
where $F=\langle x_0,x_1,\cdots,x_d \rangle$ is a free pro-$p$ group 
with $d+1$ generators $x_i$ such that 
$\pi(x_i)$ 
generates the inertia group of a prime $\widetilde{\ell}_i$ of $(\mathbb Q^{\mathrm{cyc}})_S$ lying over $\ell_i$, 
and $R=\langle r_0,r_1,\cdots,r_d \rangle_F$ is a normal subgroup of $F$ 
normally generated by $d+1$ relations $r_i$ of the form 
\[
r_i=\left\{
\begin{array}{rl}
[x_0^{-1},y_0^{-1}] & \hbox{if } i=0, \\ 
x_i^{\ell_i-1}[x_i^{-1},y_i^{-1}] & \hbox{if } 1 \le i \le d
\end{array}
\right.
\]
with $y_i \in F$ such that $\pi(y_i)$ is a Frobenius automorphism of $\widetilde{\ell}_i$ in $\widetilde{G}_S(\mathbb Q)$, and 
\[
y_i \equiv 
\left\{
\begin{array}{lll}
\underset{j=0}{\stackrel{d}{\textstyle{\prod}}} x_j^{\mathrm{lk}(\ell_i,\ell_j)} & \mod{[F,F]} & \hbox{if $p \neq 2$ or $\infty \in S$}, \\
x_0^{\mathrm{lk}(\ell_i,\ell_0)} \underset{j=1}{\stackrel{d}{\textstyle{\prod}}} x_j^{\mathrm{lk}(\ell_i,\ell_j)+\frac{\ell_j-1}{2}\mathrm{lk}(\ell_i,q)} & \mod{[F,F]} & \hbox{otherwise.} 
\end{array}
\right.
\]
\end{theorem}

On the other hand, we suppose that 
$p \neq 2$ and $k$ is an imaginary quadratic field with class number $h_k \not\equiv 0 \pmod{p}$. 
Assume that $k \neq \mathbb Q(\sqrt{-3})$ if $p=3$. 
For each prime ideal $v \in S \cup P$, 
we fix a generator $\beta_v \in O_k$ of the principal ideal $\beta_v O_k=v^{h_k}$, 
where $O_k$ denotes the ring of algebraic integers in $k$. 
For each $v \in S$, 
we fix an element $\alpha_v \in O_k \cap U_v$ such that 
$(O_k/v)^{\times}=\langle \alpha_v \bmod{v} \rangle$.  
For each $v \in P$, 
we also fix $\alpha_v \in U_v$ such that 
$\mathcal{U}_v/\widetilde{\mathcal{U}}_v = \langle \iota_v(\alpha_v)\widetilde{\mathcal{U}}_v \rangle \simeq \mathbb Z_p$, 
where $\iota_v : k_v^{\times} \rightarrow \widehat{k_v^{\times}}$ is the natural homomorphism. 
If $k_v=\mathbb Q_p$, we choose $\alpha_v=(1+p)^{-1}$ (cf.\ Example \ref{exam:kvQp}). 
For a pair $(v,w)$ of prime ideals $v, w \in S \cup P$, 
the linking number $\mathrm{lk}(v,w)$ is defined as follows: 
If $v \neq w \in S$, then $\mathrm{lk}(v,w)$ is an integer such that 
\[
\beta_v^{-1} \equiv \alpha_w^{\mathrm{lk}(v,w)} \mod{w} 
\]
and $0 \le \mathrm{lk}(v,w) < |O_k/w|$. 
If $v \neq w \in P$, then $\mathrm{lk}(v,w) \in \mathbb Z_p$ is a $p$-adic integer satisfying 
\[
\iota_w(\beta_v)^{-1} \equiv \iota_w(\alpha_w)^{\mathrm{lk}(v,w)} \mod{\widetilde{\mathcal{U}}_w} . 
\]
If $v=w$, we put $\mathrm{lk}(v,w)=0$. 

\begin{theorem}\label{thm:presk}
Assume that $p \neq 2$ and $k$ is an imaginary quadratic field with class number $h_k \not\equiv 0 \pmod{p}$, 
and that $k \neq \mathbb Q(\sqrt{-3})$ if $p=3$. 
Then $\widetilde{G}_S(k)$ has a minimal presentation
\[
\entrymodifiers={+!!<0pt,\fontdimen22\textfont2>}
\xymatrix{
1 \ar[r] & R \ar[r] & F \ar[r]^-{\pi} & \widetilde{G}_S(k) \ar[r] & 1
}
\]
where $F$ is a free pro-$p$ group with generators $\{x_v\}_{v \in S \cup P}$ such that 
$\pi(x_v)$ 
generates the inertia group of a prime $\widetilde{v}$ of $(k^{\mathrm{cyc}})_S$ lying over $v$, 
and $R$ is a normal subgroup of $F$ 
normally generated by $\{r_v\}_{v \in S \cup P}$ of the form 
\[
r_v=\left\{
\begin{array}{rl}
[x_v^{-1},y_v^{-1}] & \hbox{if } v \in P, \\ 
x_v^{|N_{k/\mathbb Q}v|-1}[x_v^{-1},y_v^{-1}] & \hbox{if } v \in S 
\end{array}
\right.
\]
with $y_v \in F$ such that $\pi(y_v)$ is a Frobenius automorphism of $\widetilde{v}$ and 
\[
y_v \equiv 
\underset{w \in S \cup P}{\textstyle{\prod}} x_w^{\mathrm{lk}(v,w)} \mod{[F,F]} . 
\]
\end{theorem}

\begin{remark}\label{rem:lkimag}
Suppose that $p$ splits in $k$ as $pO_k=\mathfrak{p}_1\mathfrak{p}_2$. 
Then $P=\{\mathfrak{p}_1, \mathfrak{p}_2\}$, $\widetilde{\mathcal{U}}_{\mathfrak{p}_i}=\{1\}$, and $\alpha_{\mathfrak{p}_i}=(1+p)^{-1}$. 
Since $\iota_{\mathfrak{p}_i}$ is injective from $U_{\mathfrak{p}_i}^{p-1}$, 
the linking number $\mathrm{lk}(v,\mathfrak{p}_i) \in \mathbb Z_p$ satisfies 
\[
\beta_v^{p-1}=(1+p)^{(p-1)\mathrm{lk}(v,\mathfrak{p}_i)} \in k_{\mathfrak{p}_i} . 
\]
Hence $\mathrm{lk}(v,\mathfrak{p}_i) \equiv 0 \pmod{p^n}$ for $1 \le n \in \mathbb Z$ 
if and only if 
$\beta_v^{p-1} \equiv 1 \pmod{\mathfrak{p}_i^{n+1}}$. 
Since an isomorphism $k_{\mathfrak{p}_1} \stackrel{\simeq}{\rightarrow} k_{\mathfrak{p}_2}$ is induced from the nontrivial element of $\mathrm{Gal}(k/\mathbb Q)$, 
we have 
$\mathrm{lk}(\mathfrak{p}_1,\mathfrak{p}_2)=\mathrm{lk}(\mathfrak{p}_2,\mathfrak{p}_1)$. 
\end{remark}

\subsection{Proof of Theorems \ref{thm:presQ} and \ref{thm:presk}.} 
Put $S'=S \setminus \{\infty\}$ if $\infty \in S$, and $S'=S$ otherwise. 
Put $S^{\ast}=S \setminus \{q\}$ if $p=2$ and $\infty \not\in S$, 
and put $S^{\ast}=S'$ otherwise. 
If $k=\mathbb Q$, 
we put $\alpha_p=(1+p)^{-1}$ or $\alpha_p=5^{-1}$ according to $p \neq 2$ or $p=2$, 
and put $\beta_{\ell}=\ell$ for each prime number $\ell \in S' \cup \{p\}$. 

Let $\iota_v : k_v^{\times} \rightarrow \widehat{k_v^{\times}}$ be the natural homomorphism for each prime $v$. 
Note that $U_v/(U_v \cap \mathrm{Ker}\iota_v)$ is the maximal pro-$p$ quotient of $U_v$ if $v \nmid \infty$. 
If $v \in S'$, the finite cyclic $p$-group $\mathcal{U}_v$ is generated by $\iota_v(\alpha_v)$, and 
\[
\iota_v(\beta_w)^{-1} = \iota_v(\alpha_v)^{\mathrm{lk}(w,v)} 
\]
for $w \in (S' \setminus \{v\}) \cup P$. 
If $v \in P$, we have 
$\mathcal{U}_v/\widetilde{\mathcal{U}}_v = \langle \iota_v(\alpha_v)\widetilde{\mathcal{U}}_v \rangle \simeq \mathbb Z_p$ 
and 
\[
\iota_v(\beta_w)^{-1} \equiv \iota_v(\alpha_v)^{\mathrm{lk}(w,v)} \mod{\widetilde{\mathcal{U}}_v}
\]
for $w \in S' \cup (P \setminus \{v\})$ 
(cf.\ Example \ref{exam:kvQp}). 
Then we obtain the following congruences in $\mathcal{J}_k \subset \prod_v \widehat{k_v^{\times}}$; 
\begin{align}
\iota_w(\beta_w)
=(\iota_w(\beta_w),(1)_{v \not\in \{w\}}) 
&\equiv (1,(\iota_v(\beta_w)^{-1})_{v \not\in \{w\}}) 
\nonumber \\
&\equiv ((1)_{v \not\in (S' \cup P)\setminus \{w\}},(\iota_v(\beta_w)^{-1})_{v \in (S' \cup P)\setminus \{w\}}) 
\nonumber \\
&\equiv \underset{v \in S' \cup P}{\textstyle{\prod}} \iota_v(\alpha_v)^{\mathrm{lk}(w,v)} \mod{\widetilde{\mathcal{U}}_S\overline{k^{\times}}}, 
\label{eqv:b_a}
\end{align}
where we note that $\beta_w \otimes 1 \in \overline{k^{\times}}$, 
and that 
$\beta_{\ell}=\ell \in \mathbb R_{>0}$, i.e., $\iota_v(\beta_{\ell})=1 \in \widehat{k_v^{\times}}$ if $v=\infty \in S$. 
If $p=2$ and $\infty \not\in S$, 
then $k=\mathbb Q$, and 
\begin{align}\label{eqv:q_a}
1 \equiv (\iota_v(-1))_v \equiv ((1)_{v \not\in S},(\iota_v(\alpha_v)^{\frac{v -1}{2}})_{v \in S})
\equiv \iota_q(\alpha_q)^{-1} \underset{v \in S^{\ast}}{\textstyle{\prod}} \iota_v(\alpha_v)^{\frac{v -1}{2}} \mod{\widetilde{\mathcal{U}}_S\overline{k^{\times}}}, 
\end{align}
where we note that $-1 \in \widetilde{\mathcal{U}}_p$ (cf.\ Example \ref{exam:kvQp}), 
and that $\frac{q-1}{2} \equiv -1 \pmod{2}$, i.e., $\mathcal{U}_q \simeq \mathbb Z/2\mathbb Z$. 
Note that $\beta_v \varpi_v^{-h_k} \in U_v$ for arbitrary uniformizer $\varpi_v$ of $k_v$, 
where $h_k=1$ if $k=\mathbb Q$. 
Since $h_k \in \mathbb Z_p^{\times}$, 
we have $\widehat{k_v^{\times}}/\mathcal{U}_v = \langle \iota_v(\beta_v) \mathcal{U}_v \rangle \simeq \mathbb Z_p$ 
for $v \in S' \cup P$. 

Recall that $\widetilde{G}_{S,v}$ is identified with the decomposition subgroup of $\widetilde{G}_S=\widetilde{G}_S(k)$ for each fixed prime $\widetilde{v}|v$ of $(k^{\mathrm{cyc}})_S$. 
Let $\widetilde{I}_v \subset \widetilde{G}_{S,v}$ be the inertia subgroup. 
If $v \not\in P$ and $v \nmid \infty$, 
then the inertia subgroup of $\overline{G}_v \simeq \mathbb Z_p \rtimes \mathbb Z_p$ is isomorphic to $\mathbb Z_p$ 
(cf.\ e.g.\ \cite[Theorem 10.2]{Koch}), 
and hence $\widetilde{I}_v$ is also procyclic. 
If $v \in P$, then $(k^{\mathrm{cyc}})_Sk_v \subset k_v^{\mathrm{ur},p} k_v^{\mathrm{cyc}}$, and hence 
\[
\widetilde{I}_v = \mathrm{Gal}((k^{\mathrm{cyc}})_Sk_v/(k^{\mathrm{cyc}})_Sk_v \cap k_v^{\mathrm{ur},p}) 
\simeq \mathrm{Gal}(k_v^{\mathrm{ur},p} k_v^{\mathrm{cyc}}/k_v^{\mathrm{ur},p}) 
\simeq \mathbb Z_p . 
\]
Proposition \ref{prop:idele} implies the existence of 
a Frobenius element $\sigma_v \in \widetilde{G}_{S,v}$ 
and a generator $\tau_v$ of procyclic $\widetilde{I}_v$ for $v \in S' \cup P$ 
such that 
\begin{align}\label{eq:ba_st}
\mathrm{rec}_S(\iota_v(\beta_v)\widetilde{\mathcal{U}}_S\overline{k^{\times}}) 
=\sigma_v^{h_k} [\widetilde{G}_S,\widetilde{G}_S] , 
\quad 
\mathrm{rec}_S(\iota_v(\alpha_v)\widetilde{\mathcal{U}}_S\overline{k^{\times}}) 
=\tau_v^{h_k} [\widetilde{G}_S,\widetilde{G}_S] . 
\end{align}
Since $k$ has no $p$-extensions unramified outside $S \setminus S^{\ast}$, 
the pro-$p$ group $\widetilde{G}_S$ is generated by $\{\tau_v\}_{v \in S^{\ast} \cup P}$. 
Since $\widetilde{d}_S=|S^{\ast} \cup P|$ by Remark \ref{rem:trivB} and Proposition \ref{prop:genrk}, 
$\widetilde{G}_S$ has a minimal presentation 
\[
\entrymodifiers={+!!<0pt,\fontdimen22\textfont2>}
\xymatrix{
1 \ar[r] & R \ar[r] & F \ar[r]^-{\pi} & \widetilde{G}_S \ar[r] & 1 
}
\]
with a free pro-$p$ group $F$ generated by $\{x_v\}_{v \in S^{\ast} \cup P}$ 
such that $\pi(x_v)=\tau_v$. 
By (\ref{eqv:b_a}), (\ref{eqv:q_a}) and (\ref{eq:ba_st}), 
there is an element $y_w \in F$ for each $w \in S^{\ast} \cup P$ such that 
$\pi(y_w)=\sigma_w$ and 
\[
y_w \equiv 
\left\{
\begin{array}{lll}
\underset{v \in S' \cup P}{\textstyle{\prod}} x_v^{\mathrm{lk}(w,v)} & \mod{[F,F]} & \hbox{if $p \neq 2$ or $\infty \in S$}, \\
\Big(\, \underset{v \in S^{\ast}}{\textstyle{\prod}} x_v^{\frac{v -1}{2}} \Big)^{\mathrm{lk}(w,q)} 
\underset{v \in S^{\ast} \cup P}{\textstyle{\prod}} x_v^{\mathrm{lk}(w,v)} 
& \mod{[F,F]} & \hbox{otherwise}. 
\end{array}
\right.
\]
For $v \in S^{\ast} \cup P$, we put 
\begin{align}\label{peripheralH}
H_v=
\left\{
\begin{array}{ll}
\overline{G}_v \simeq \mathbb Z_p \rtimes \mathbb Z_p & \hbox{if $v \in S^{\ast}$}, \\
G_v^{\mathrm{cr}} \simeq \mathbb Z_p \times \mathbb Z_p & \hbox{if $v \in P$}, 
\end{array}
\right.
\end{align}
as in the proof of Theorem \ref{thm:relrk}, 
and let 
$F_v$ be a free pro-$p$ group generated by $\{s_v,t_v\}$. 
Then there exists a commutative diagram 
\[
\entrymodifiers={+!!<0pt,\fontdimen22\textfont2>}
\xymatrix{
1 \ar[r] & R \ar[r] & F \ar[r]^-{\pi} & \widetilde{G}_S \ar[r] & 1 \\
1 \ar[r] & R_v \ar[r] \ar[u] & F_v \ar[r]^-{\pi_v} \ar[u]^-{\chi_v} & H_v \ar[r] \ar[u]^-{\varphi_v} & 1 
}
\]
with minimal presentations of $H_v$, 
such that 
$\chi_v(t_v)=x_v$, $\chi_v(s_v)=y_v$, $\varphi_v(H_v)=\widetilde{G}_{S,v}$, 
and $\pi_v(s_v)$ (resp.\ $\pi_v(t_v)$) is a Frobenius element (resp.\ a generator of the inertia subgroup) of $H_v$, 
where $R_v$ is a normal subgroup of $F_v$ normally generated by 
either 
$t_v^{|N_{k/\mathbb Q}v|-1} [t_v^{-1},s_v^{-1}]$ or $[t_v^{-1},s_v^{-1}]$ 
according to $v \in S^{\ast}$ or $v \in P$ (cf.\ \cite[Theorem 10.2]{Koch}). 
Since $\Sha^2(\widetilde{G}_S)=\{0\}$ by Remark \ref{rem:trivB} and Theorem \ref{thm:relrk}, 
the localization map 
\[
H^2(\widetilde{G}_S) \stackrel{\mathrm{loc}}{\longrightarrow} \underset{v \in S^{\ast} \cup P}{\textstyle{\bigoplus}} H^2(H_v) 
\]
(which is induced from $\{\varphi_v\}_{v \in S^{\ast} \cup P}$) is injective, 
and hence 
$R$ is normally generated by 
\[
\{\, x_v^{|N_{k/\mathbb Q}v|-1} [x_v^{-1},y_v^{-1}] \,\}_{v \in S^{\ast}} 
\quad \hbox{and} \quad 
\{\, [x_v^{-1},y_v^{-1}] \,\}_{v \in P} 
\]
as a normal subgroup of $F$ (cf.\ \cite[Theorem 6.14]{Koch}). 
Thus the proof of Theorems \ref{thm:presQ} and \ref{thm:presk} is completed.

\begin{remark}
Suppose that $p=2$, $\infty \in S$ and $\ell_d \equiv 3 \pmod{4}$ in the case of Theorem \ref{thm:presQ}. 
If we put $\alpha_{\infty}=-1$ and $H_{\infty}=\overline{G}_{\infty} \simeq \mathbb Z/2\mathbb Z$, 
then 
\[
1 \equiv (\iota_v(-1))_v 
\equiv \iota_{\ell_d}(\alpha_{\ell_d})^{-1} 
\iota_{\infty}(\alpha_{\infty})
\underset{v \in S \setminus \{\infty,\ell_d\}}{\textstyle{\prod}} \iota_v(\alpha_v)^{\frac{v -1}{2}} \mod{\widetilde{\mathcal{U}}_S\overline{k^{\times}}} , 
\]
and hence the same arguments using $S \setminus \{\ell_d\}$ instead of $S^{\ast}$ 
yield the existence of $x_{\infty}$ and $\{\eta_i\}_{0 \le i <d}$ such that 
$1 \neq \pi(x_{\infty}) \in \widetilde{I}_{\infty}=\widetilde{G}_{S,\infty} \simeq \mathbb Z/2\mathbb Z$, 
\begin{align*}
x_d &\equiv x_{\infty}
\underset{j=1}{\stackrel{d-1}{\textstyle{\prod}}} x_j^{\frac{\ell_j -1}{2}} \mod{[F,F]} , \\
\eta_i &\equiv x_0^{\mathrm{lk}(\ell_i,\ell_0)} x_{\infty}^{\mathrm{lk}(\ell_i,\ell_d)} 
\underset{j=1}{\stackrel{d-1}{\textstyle{\prod}}} x_j^{\mathrm{lk}(\ell_i,\ell_j)+\frac{\ell_j -1}{2}\mathrm{lk}(\ell_i,\ell_d)} \mod{[F,F]} , 
\end{align*}
and $R$ is normally generated by 
$\{\, x_i^{\ell_i -1} [x_i^{-1},\eta_i^{-1}] \,\}_{1 \le i <d}$, 
$[x_0^{-1},\eta_0^{-1}]$ and $x_{\infty}^2$. 
\end{remark}

\begin{remark}\label{rem:gamma}
Let $\overline{\gamma}$ be the element of $\mathrm{Gal}(k(\mu_{p^{\infty}})/k)$ such that $\overline{\gamma}(\zeta)=\zeta^{\kappa}$ for any $\zeta \in \mu_{p^{\infty}}$, 
where $\kappa=1+p$ or $\kappa=5$ according to $p \neq 2$ or $p=2$, 
and $\mu_{p^{\infty}}$ denotes the group of $p$-power roots of unity. 
If $k_v=\mathbb Q_p$ for $v \in P$, 
there is a commutative diagram 
\[
\entrymodifiers={+!!<0pt,\fontdimen22\textfont2>}
\xymatrix{
k_v^{\times} \ar[r]^-{\iota_v \bmod{\widetilde{\mathcal{U}}_S \overline{k^{\times}}}} \ar[d]^-{\rho} & \mathcal{J}_k/\widetilde{\mathcal{U}}_S \overline{k^{\times}} \ar[r]^-{\mathrm{rec}_S} & \widetilde{G}_S^{\mathrm{ab}} \ar[r]^-{|_{k^{\mathrm{cyc}}}} & \mathrm{Gal}(k^{\mathrm{cyc}}/k) \\
\mathrm{Gal}(\mathbb Q_p(\mu_{p^{\infty}})/\mathbb Q_p) \ar[rrr]^-{\simeq} & & & \mathrm{Gal}(k(\mu_{p^{\infty}})/k) \ar[u]^-{|_{k^{\mathrm{cyc}}}} \\
} 
\]
with the reciprocity map $\rho$ of local class field theory, which satisfies 
$\rho(\alpha_v)(\zeta)=\zeta^{\alpha_v^{-1}}=\zeta^{\kappa}$ for any $\zeta \in \mu_{p^{\infty}}$ (cf.\ e.g.\ \cite[II, Exercise 3.4.3]{Gra}). 
Then we have $\pi(x_v)^{h_k}|_{k^{\mathrm{cyc}}}=\overline{\gamma}|_{k^{\mathrm{cyc}}}$. 
In particular, 
$\pi(x_0)|_{\mathbb Q^{\mathrm{cyc}}}=\overline{\gamma}|_{\mathbb Q^{\mathrm{cyc}}}$ 
for $x_0 \in F$ of Theorem \ref{thm:presQ}. 
\end{remark}

\begin{remark}\label{rem:anlgG}
A link group $\pi_1(\mathsf{X})$ is the fundamental group of the complement $\mathsf{X}$ of the open tubular neighborhood $\bigsqcup_{i=1}^{d} V_{\mathsf{K}_i}^{\circ}$ of a link $\mathsf{L}=\bigsqcup_{i=1}^{d} \mathsf{K}_i$ in a rational homology $3$-sphere $\mathsf{M}$. 
The image of $\pi_1(\partial V_{\mathsf{K}_i}) \simeq \mathbb Z \times \mathbb Z$ in $\pi_1(\mathsf{X})$ is analogous to 
the decomposition group of a ramified prime $v \in S^* \cup P$ in $\widetilde{G}_S(k)$, 
which is a quotient of the local Galois group $H_v \simeq \mathbb Z_p \rtimes \mathbb Z_p$ (cf.\,(\ref{peripheralH})). 
Hence the Koch type presentation of $\widetilde{G}_S(k)$ is 
analogous to a Milnor presentation of $\pi_1(\mathsf{X})$ (cf.\ e.g.\ \cite{Mor}). 
\end{remark}

\subsection{Preliminaries for consequences.} 
Let $G$ be a pro-$p$ group. 
Put $G_1=G$, 
and put $G_n=[G_{n-1},G]$ for $2 \le n \in \mathbb Z$ recursively. 
Then $\{G_n\}_{n \ge 1}$ is the lower central series of $G$. 
Put $G_{(n)}=\{g \in G\,|\,g-1 \in (I_G)^n\}$ for $1 \le n \in \mathbb Z$, 
where $I_G=\mathrm{Ker}(\mathbb F_p[[G]] \rightarrow \mathbb F_p)$ 
is the augmentation ideal of $\mathbb F_p[[G]]$. 
Then $G_{(1)}=G$ and 
$\{G_{(n)}\}_{n \ge 1}$ is the Zassenhaus filtration of $G$. 
Put $[g_1,g_2,g_3]=[[g_1,g_2],g_3] \in G_3$ for $g_1,g_2,g_3 \in G$. 

Let $F=\langle x_0,x_1,\cdots,x_d \rangle$ be a free pro-$p$ group 
generated by $\{x_i\}_{0 \le i \le d}$. 
Let $\varepsilon_{\mathbb Z_p[[F]]} : \mathbb Z_p[[F]] \rightarrow \mathbb Z_p$ be the augmentation map, and $\frac{\partial}{\partial x_i} : \mathbb Z_p[[F]] \rightarrow \mathbb Z_p[[F]]$ be the pro-$p$ Fox derivative. 
For a multi-index $I=(i_1 \cdots i_n)$ with $0 \le i_1,\cdots,i_n \le d$ and $y \in F$, we put 
\[
\varepsilon_{I}(y)=
\varepsilon_{\mathbb Z_p[[F]]}
\Big( \frac{\partial^n y}{\partial x_{i_1} \cdots \partial x_{i_n}} \Big) \in \mathbb Z_p 
\]
which is the coefficient of $X_{i_1} \cdots X_{i_n}$ 
in the expansion of $y$ by the pro-$p$ Magnus isomorphism 
\[
\hat{M} : \mathbb Z_p[[F]] \simeq \mathbb Z_p \langle\!\langle X_1,\cdots,X_d \rangle\!\rangle : x_j \mapsto 1+X_j 
\]
(cf.\ \cite[Proposition 8.14]{Mor}). 
We also put 
\[
\varepsilon_{I,p}(y) = \varepsilon_{I}(y)+p\mathbb Z_p \in \mathbb F_p 
\]
the mod $p$ Magnus coefficient. 

Suppose that $F$ is the free pro-$p$ group in Theorem \ref{thm:presQ}, 
and let $y_i \in F$ be the element obtained in Theorem \ref{thm:presQ}. 
Then $\pi(y_i)$ is an analogue of the longitude of a component of a link, 
and the mod $p$ Milnor number 
\[
\mu_p(i_1 \cdots i_n i)=\varepsilon_{I,p}(y_i)
\]
for a multi-index $(i_1 \cdots i_n i)$ is defined.

\subsection{R\'edei symbols.} 
We consider the case where $p=2$ and $S=\{\ell_1,\cdots,\ell_d,\infty\}$. 
Put $\ell_0^{\ast}=\ell_0=2$, and put $\ell_i^{\ast}=(-1)^{\frac{\ell_i -1}{2}} \ell_i \equiv 1 \pmod{4}$ for $1 \le i \le d$. 
For $0 \le i \le d$, 
the elements $y_i$ in Theorem \ref{thm:presQ} can be written in the form 
\begin{align}\label{def:ciab}
y_i \equiv x_d^{\mathrm{lk}(\ell_i,\ell_d)} \cdots x_1^{\mathrm{lk}(\ell_i,\ell_1)} x_0^{\mathrm{lk}(\ell_i,\ell_0)}
\underset{a < b}{\textstyle{\prod}} [x_a,x_b]^{c_{iab}} \mod{F_3}
\end{align}
with some $c_{iab}=-c_{iba} \in \mathbb Z_2$, 
where all pairs $(a,b)$ with $0 \le a < b \le d$ run in the product. 
Assume that there are distinct prime numbers $\ell_a$, $\ell_b \in S \cup \{2\}$ satisfying 
\begin{align}\label{eqv:abba}
\mathrm{lk}(\ell_a,\ell_b) \equiv \mathrm{lk}(\ell_b,\ell_a) \equiv 0 \pmod{2} . 
\end{align}
For such $\ell_a$ and $\ell_b$, the quadratic field $\mathbb Q(\sqrt{\ell_a^{\ast}\ell_b^{\ast}})$ has a unique cyclic extension $K_{(a,b)}$ of degree $4$ unramified outside $\infty$ (cf.\ e.g.\ \cite{Yam84}). 
Then $\mathbb Q(\sqrt{\ell_a^{\ast}}, \sqrt{\ell_b^{\ast}}) \subset K_{(a,b)}$, 
and $K_{(a,b)}/\mathbb Q$ is the R\'edei extension, 
which is a dihedral extension of degree $8$ unramified outside $\{\ell_a,\ell_b,\infty\}$. 
Moreover if a prime number $\ell_i$ ($0 \le i \le d$) satisfies 
\begin{align}\label{eqv:iaib}
\mathrm{lk}(\ell_i,\ell_a) \equiv \mathrm{lk}(\ell_i,\ell_b) \equiv 0 \pmod{2} , 
\end{align}
then a prime ideal $\mathfrak{l}_i$ of $\mathbb Q(\sqrt{\ell_a^{\ast}\ell_b^{\ast}})$ lying over $\ell_i$ splits in $\mathbb Q(\sqrt{\ell_a^{\ast}}, \sqrt{\ell_b^{\ast}})$, 
and the R\'edei symbol (cf.\ \cite{Red39}) for such triple $(\ell_a,\ell_b,\ell_i)$ can be defined as 
\[
{[\ell_a,\ell_b,\ell_i]}=
{[\ell_b,\ell_a,\ell_i]}=
\left\{
\begin{array}{rl}
1 & \hbox{if $\mathfrak{l}_i$ splits completely in $K_{(a,b)}$,}\\
-1 & \hbox{otherwise.}
\end{array}
\right.
\]
The following proposition relates $c_{iab}$, $[\ell_a,\ell_b,\ell_i]$, and the mod $2$ Milnor number $\mu_2(abi)$. 

\begin{proposition}\label{Redei:ciab}
Under the settings above, we have 
\[
{[\ell_a,\ell_b,\ell_i]}=(-1)^{c_{iab}}=(-1)^{\mu_2(abi)}
\]
if (\ref{eqv:abba}) and (\ref{eqv:iaib}) are satisfied. 
\end{proposition}

\begin{proof}
Let 
\[
\entrymodifiers={+!!<0pt,\fontdimen22\textfont2>}
\xymatrix{
1 \ar[r] & R \ar[r] & F \ar[r]^-{\pi} & \widetilde{G}_S(\mathbb Q) \ar[r] & 1
}
\]
be the presentation of $\widetilde{G}_S(\mathbb Q)$ obtained in Theorem \ref{thm:presQ}. 
Let $N_{(a,b)}$ be the minimal normal subgroup of $F$ 
including $\{x_j|j \not\in \{a,b\} \} \cup \{x_a^2,x_b^2\}$. 
By Theorem \ref{thm:presQ} and (\ref{eqv:abba}), we have $y_a$, $y_b \in F_2N_{(a,b)}$, and hence $r_a$, $r_b \in F_3N_{(a,b)}$. 
Since $r_j \in N_{(a,b)}$ for any $j \not\in \{a,b\}$, 
we have $R \subset F_3N_{(a,b)}$. 
Since 
\[
(x_ax_b)^2 \equiv [x_a,x_b] \mod{N_{(a,b)}}, 
\]
$F/F_2N_{(a,b)}$ is an elementary abelian $2$-group of order at most $4$. 
Since 
\[
{[x_a,x_b]}^2 \equiv [x_a^2,x_b] \equiv 1 \mod{F_3N_{(a,b)}}, 
\]
$F_2N_{(a,b)}/F_3N_{(a,b)}$ is a cyclic group generated by $[x_a,x_b]F_3N_{(a,b)}$ of order at most $2$. 
Recall that the ramification indices of any ramified primes are $2$ 
in the dehedral extension $K_{(a,b)}/\mathbb Q$ of order $8$ 
which is unramified outside $\{\ell_a,\ell_b,\infty\}$. 
Hence $\pi$ induces an isomorphism  
\[
F/F_3N_{(a,b)} \simeq \mathrm{Gal}(K_{(a,b)}/\mathbb Q) : xF_3N_{(a,b)} \mapsto \pi(x)|_{K_{(a,b)}} .
\]
By (\ref{def:ciab}) and (\ref{eqv:iaib}), we have 
\begin{align*}
y_i \equiv [x_a,x_b]^{c_{iab}} \mod{F_3N_{(a,b)}} . 
\end{align*}

If $i \not\in \{a,b\}$, 
$\ell_i$ splits completely in $\mathbb Q(\sqrt{\ell_a^{\ast}}, \sqrt{\ell_b^{\ast}})$ by (\ref{eqv:iaib}). 
Then $\pi(y_i)|_{K_{(a,b)}}$ generates the decomposition group of any primes lying over $\ell_i$ by Theorem \ref{thm:presQ}, 
and hence we have ${[\ell_a,\ell_b,\ell_i]}=(-1)^{c_{iab}}$. 
Suppose that $i \in \{a,b\}$. 
Then $\ell_i$ splits in $\mathbb Q(\sqrt{\ell_j^{\ast}})$ for $i \neq j \in \{a,b\}$ by $(\ref{eqv:abba})$. 
By the definition of R\'edei symbol, 
we have $[\ell_a,\ell_b,\ell_i]=-1$ if and only if 
$\mathbb Q(\sqrt{\ell_j^{\ast}})$ is equal to the decomposition field of 
any primes of $K_{(a,b)}$ lying over $\ell_i$. 
By Theorem \ref{thm:presQ}, 
the elements $\pi(x_i)|_{K_{(a,b)}}$, $\pi(y_i)|_{K_{(a,b)}}$ of order at most $2$ generate the decomposition group of the prime of $K_{(a,b)}$ which is the restriction of $\widetilde{\ell}_i$. 
Hence we have ${[\ell_a,\ell_b,\ell_i]}=(-1)^{c_{iab}}$. 

Put $h_i=[x_a,x_b]^{-c_{iab}} y_i \in F_3N_{(a,b)}$. 
By the basic properties of mod $2$ Magnus coefficients (cf.\ e.g.\ \cite[Proposition 2.18]{Vog05}), 
we have 
\begin{align*}
&\varepsilon_{(ab),2}([x_a,x_b]) =1 , \\
&\varepsilon_{(ab),2}(y_i) =
\varepsilon_{(ab),2}([x_a,x_b]^{c_{iab}})
+\varepsilon_{(a),2}([x_a,x_b]^{c_{iab}}) \varepsilon_{(b),2}(h_i)
+\varepsilon_{(ab),2}(h_i) ,
\end{align*}
and 
$\varepsilon_{I,2}(g^{-1}hh'g)=0$ 
for any $g \in F$ and $I \in \{(a),(b),(ab)\}$ 
if $\varepsilon_{I,2}(h)=\varepsilon_{I,2}(h')=0$ for all $I \in \{(a),(b),(ab)\}$. 
Since 
$\varepsilon_{I}(F_3)=0$ (cf.\ e.g.\ \cite[Proposition 8.15]{Mor}) 
and 
$\varepsilon_{I,2}(x_j^{\pm 1})=\varepsilon_{I,2}(x_a^{\pm 2})=\varepsilon_{I,2}(x_b^{\pm 2})=0$ 
for any $I \in \{(a),(b),(ab)\}$ and $j \not\in\{a,b\}$, 
the continuity of $\varepsilon_{I,2} : F \rightarrow \mathbb F_2$ yields that 
$\varepsilon_{I,2}(F_3N_{(a,b)})=0$. 
In particular, $\varepsilon_{(b),2}(h_i)=\varepsilon_{(ab),2}(h_i)=0$. 
Therefore 
\[
\mu_2(abi) = 
\varepsilon_{(ab),2}(y_i) =
c_{iab}+2\mathbb Z_2 . 
\]
Thus the proof of Proposition \ref{Redei:ciab} is completed. 
\end{proof}

The following proposition yields that 
the R\'edei symbols $[\ell_a,\ell_b,\ell_i]$ with $i \in \{a,b\}$ 
are written by the quartic residue symbols $\left(\frac{\phantom{z}}{\phantom{\ell}}\right)_4$ 
defined as follows; 
$\left(\frac{z}{\ell}\right)_4 = \pm 1 \equiv z^{\frac{\ell -1}{4}} \pmod{\ell}$ for a prime number $\ell \equiv 1 \pmod{4}$ and $z \in \mathbb Z_{\ell}^{\times}$ such that $\left(\frac{z}{\ell}\right)=1$, 
and $\left(\frac{z}{2}\right)_4 = (-1)^{\frac{z-1}{8}}$ for $z \in \mathbb Z$ such that $z \equiv 1 \pmod{8}$. 

\begin{proposition}[{cf.\ \cite{Red39}}]\label{Redei:4th}
Under the settings above, we have the following equations 
for distinct prime numbers $\ell_a$, $\ell_b \in S \cup \{2\}$ satisfying (\ref{eqv:abba}); 
\begin{align*}
&{[\ell_a,\ell_b,\ell_a]}
={\textstyle \big(\frac{\ell_b}{\ell_a}\big)_4}
\hbox{ if $\ell_a \not\equiv 3 \pmod{4}$ and $\ell_b \not\equiv 3 \pmod{4}$,} \\
&{[\ell_a,\ell_b,\ell_a]}={[\ell_a,\ell_b,\ell_b]}
={\textstyle \big(\frac{-\ell_b}{\ell_a}\big)_4}
\hbox{ if $\ell_a \not\equiv 3 \pmod{4}$ and $\ell_b \equiv 3 \pmod{4}$.} 
\end{align*}
\end{proposition}

\begin{proof}
See \cite[(53)--(53'')]{Red39}. 
Alternatively, 
since ${[\ell_a,\ell_b,\ell_a]}=-1$ if and only if the narrow class number of $\mathbb Q(\sqrt{\ell_a^{\ast}\ell_b^{\ast}})$ is not divisible by $8$ and the narrow ideal class of $\mathfrak{l}_a$ is nontrivial, 
the statement is obtained as a translation of well known results on the narrow class groups 
(cf.\ e.g.\ \cite[Propositions 3.3--3.6]{Yam84}). 
\end{proof}

As a special case of Theorem \ref{thm:presQ}, 
we obtain the following result similar to \cite[Theorem 3.12]{Vog05}.

\begin{corollary}\label{cor:free4}
Suppose that $k=\mathbb Q$, $\ell_0=p=2$ and $S=\{\ell_1,\cdots,\ell_d,\infty\}$. 
Assume that $\ell_i \equiv 1 \pmod{4}$ for all $1 \le i \le d$, 
and that 
\[
\mathrm{lk}(\ell_a,\ell_b) \equiv \mathrm{lk}(\ell_a,\ell_b) \equiv 0 \pmod{2}
\]
for all pairs $(a,b)$ with $0 \le a, b \le d$. 
Then $d+1$ relations $r_i$ of the presentation of $\widetilde{G}_S(\mathbb Q)$ 
in Theorem \ref{thm:presQ} are written in the form 
\[
r_i \equiv \underset{a < b}{\textstyle{\prod}} [x_a,x_b,x_i]^{\mu_2(abi)}
\cdot \stackrel{i-1}{\underset{a=0}{\textstyle{\prod}}} [x_a,x_i,x_a]^{\mu_2(aia)}
\cdot \stackrel{d}{\underset{b=i+1}{\textstyle{\prod}}} [x_i,x_b,x_b]^{\mu_2(ibb)}
\mod{F_{(4)}}
\]
for each $0 \le i \le d$, where $F_{(4)}$ denotes the $4$th step of the Zassenhaus filtration of $F$. 
\end{corollary}

\begin{proof}
Note that $F_{(2)}=F^2F_2$, $F_{(3)}=F^4F_2^2F_3$ and $F_{(4)}=F^4F_2^2F_4$ (cf.\ \cite[Theorems 11.2 and 12.9]{DDMS}). 
By (\ref{def:ciab}), we have 
\[
y_i^{-1} \equiv 
x_0^{2e_{i0}}
x_1^{2e_{i1}}
\cdots 
x_d^{2e_{id}} 
\underset{a < b}{\textstyle{\prod}} [x_a,x_b]^{c_{iab}}
\mod{F_{(3)}}
\]
with some $e_{ij} \in \{0,1\}$ such that $2e_{ij} \equiv \mathrm{lk}(\ell_i,\ell_j) \pmod{4}$. 
Then $e_{ii}=0$ and 
\[
(-1)^{e_{ij}}={\textstyle \big(\frac{\ell_i}{\ell_j}\big)_4}=[\ell_i,\ell_j,\ell_j]
\]
for $j \neq i$ by Proposition \ref{Redei:4th}. 
Since $[F_{(n)},F_{(4-n)}] \subset F_{(4)}$ for $n \in \{1,2\}$, a direct calculation shows that 
\[
r_i \equiv {[x_i^{-1},y_i^{-1}]} \equiv 
\underset{a < b}{\textstyle{\prod}} [x_a,x_b,x_i]^{c_{iab}}
\cdot \stackrel{d}{\underset{j=0}{\textstyle{\prod}}} [x_i,x_j,x_j]^{e_{ij}}
\mod{F_{(4)}}. 
\]
By Proposition \ref{Redei:ciab}, we obtain the claim. 
\end{proof}

\begin{example}\label{ex:Borromean}
Put $S=\{\ell_1,\ell_2,\infty\}$ with $\ell_1=113$ and $\ell_2=593$. 
Then $\mathrm{lk}(\ell_a,\ell_b) \equiv 0 \pmod{4}$ for any $0 \le a,b \le 2$. 
By Proposition \ref{Redei:ciab} and PARI/GP \cite{pari}, 
one can see that $\mu_2(abi)=1$ if $abi$ is a permutation of $012$, 
and $\mu_2(abi)=0$ otherwise. 
Then 
\begin{align*}
r_0 \equiv [x_1,x_2,x_0], \quad 
r_1 \equiv [x_2,x_0,x_1], \quad 
r_2 \equiv [x_0,x_1,x_2] \mod{F_{(4)}}
\end{align*}
by Corollary \ref{cor:free4}, 
and hence $(2,113,593)$ is also a triple of (proper) Borromean primes modulo $2$ 
in the sense of \cite{Mor,Vog05}. 
%? Borromean()
%[113, 593]
%[337, 353]
%[577, 593]
%? Free4()
%[337, 593]
%? 
\end{example}

\begin{example}
For $S=\{\ell_1,\ell_2,\infty\}$ with $\ell_1=337$ and $\ell_2=593$, 
one can see that $\mathrm{lk}(\ell_a,\ell_b) \equiv 0 \pmod{4}$ and 
$\mu_2(abi)=0$ for any $0 \le a,b,i \le 2$ 
by Proposition \ref{Redei:ciab} and PARI/GP \cite{pari}. 
Then 
\[
\widetilde{G}_S(\mathbb Q)/\widetilde{G}_S(\mathbb Q)_{(4)}
\simeq F/F_{(4)}
\]
by Corollary \ref{cor:free4}. 
\end{example}

\subsection{Mild pro-$p$ groups.}\label{sbsec:mild} 
A finitely presented pro-$p$ group $G$ is said to be `mild' 
(with respect to the Zassenhaus filtration) 
when 
$G$ has a presentation $F/R \simeq G$ with a system of relations which make a `strongly free sequence' in the graded Lie algebra $\mathrm{gr}\,F=\bigoplus_{n \ge 1} F_{(n)}/F_{(n+1)}$ (cf.\ \cite{Lab06}). 
Labute \cite{Lab06} gave the first example of mild $G_S(\mathbb Q)$, 
in particular having the cohomological dimension $cd(G_S(\mathbb Q))=2$, 
by showing that $G_S(\mathbb Q)$ is mild if $S$ is a `circular set'. 
Such criteria for mildness have been reformulated as a `cup-product criterion' (cf.\ \cite[Theorem 5.5]{Sch07} and \cite{For11,Gar15,LM11}), 
which induces the existence of mild $\widetilde{G}_S(k)$ in various situations (cf.\ \cite{BLM13}). 
We also obtain circular sets of primes including $\ell_0=p=2$ as follows. 
%We also obtain the following sufficient condition for the mildness of $\widetilde{G}_S(\mathbb Q)$. 

\begin{theorem}\label{thm:mild}
Suppose that $k=\mathbb Q$, $S^{\ast} =\{\ell_1,\cdots,\ell_d\}$, 
\[
S=\left\{
\begin{array}{ll}
S^{\ast} & \hbox{if } p \neq 2, \\ 
S^{\ast} \cup \{q\} & \hbox{if } p=2, 
\end{array}
\right.
\]
and $d=|S^{\ast}|>1$ is odd, 
where $q \equiv 3 \pmod{4}$ and $q ,\infty \not\in S^{\ast}$. 
Put $\ell_0=p$. 
Assume that $\{p\} \cup S^{\ast}$ is a `circular set', i.e., 
there is a bijection $\sigma : \mathbb Z/(d+1)\mathbb Z \rightarrow \{0,1,\cdots,d\}$ satisfying the following conditions: 
\begin{enumerate}

\item 
$\ell_{\sigma(i)} \not\equiv 3 \pmod{4}$ if $p=2$ and $i$ is even, 

\item 
$\mathrm{lk}(\ell_{\sigma(i)},\ell_{\sigma(j)}) \equiv 0 \pmod{p}$ if $i$ and $j$ are even, 

\item 
$\underset{i=0}{\stackrel{d}{\prod}} 
\widetilde{\mathrm{lk}}(\ell_{\sigma(i)},\ell_{\sigma(i+1)}) 
\not\equiv 
\ \underset{i=0}{\stackrel{d}{\prod}} 
\widetilde{\mathrm{lk}}(\ell_{\sigma(i+1)},\ell_{\sigma(i)})
\pmod{p}$, 
where 
\[
\widetilde{\mathrm{lk}}(\ell_i,\ell_j)=\left\{
\begin{array}{ll}
\mathrm{lk}(\ell_i,\ell_j) 
& \hbox{if } p \neq 2, \\ 
\mathrm{lk}(\ell_i,\ell_j) + \delta_{\ell_j+4\mathbb Z,3+4\mathbb Z} \mathrm{lk}(\ell_i,q) 
& \hbox{if } p=2. 
\end{array}
\right.
\]

\end{enumerate}
Then $\widetilde{G}_S(\mathbb Q)$ is a mild pro-$p$ group 
(with respect to the Zassenhaus filtration) of deficiency zero. 
In particular, the cohomological dimension $cd(\widetilde{G}_S(\mathbb Q)) =2$, 
the Euler-Poincar\'e characteristic $\chi(\widetilde{G}_S(\mathbb Q))=1$, 
and $\widetilde{G}_S(\mathbb Q)$ is not $p$-adic analytic. 
\end{theorem}

\begin{proof}
The pro-$p$ group $\widetilde{G}_S(\mathbb Q)$ is mild if 
the $\mathbb F_p$-vector space 
$H^1(\widetilde{G}_S(\mathbb Q))$ has a decomposition 
$H^1(\widetilde{G}_S(\mathbb Q))=U \oplus V$ 
with the subspaces $U$, $V$ such that 
\[
U \cup V =H^2(\widetilde{G}_S(\mathbb Q)) \quad \hbox{and} \quad V \cup V =\{0\} 
\]
(cf.\ e.g.\ \cite[Cup-product criterion]{Gar15}), 
where $U \cup V$ (resp.\ $V \cup V$) denotes the image of $U \otimes V$ (resp.\ $V \otimes V$) by the cup product 
\[
\cup : H^1(\widetilde{G}_S(\mathbb Q)) \otimes H^1(\widetilde{G}_S(\mathbb Q)) 
\rightarrow H^2(\widetilde{G}_S(\mathbb Q)). 
\]
Let 
\[
\entrymodifiers={+!!<0pt,\fontdimen22\textfont2>}
\xymatrix{
1 \ar[r] & R \ar[r] & F \ar[r]^-{\pi} & \widetilde{G}_S(\mathbb Q) \ar[r] & 1
}
\]
be the minimal presentation of $\widetilde{G}_S(\mathbb Q)$ 
obtained in Theorem \ref{thm:presQ}. 
Put 
\[
b_{i,j}=\widetilde{\mathrm{lk}}(\ell_i,\ell_j)+p\mathbb Z_p \in \mathbb F_p
\]
for $0 \le i,j \le d$. 
Then the normal subgroup $R$ of $F=\langle x_0,\cdots,x_d \rangle$ is normally generated by $d+1$ relations 
$r_i=x_i^{(\ell_i-1)(1-\delta_{i,0})}[x_i^{-1},y_i^{-1}]$ $(0 \le i \le d)$ 
with $y_i$ written in the form 
\[
y_i \equiv 
\underset{j=0}{\stackrel{d}{\textstyle{\prod}}} x_j^{b_{i,j}} \mod{F_{(2)}}. 
\]
Let $\chi_i \in H^1(F)$ be the dual element of $x_i$, i.e., $\chi_i(x_j)=\delta_{i,j}$ for $0 \le j \le d$. 
Then $\chi_0,\cdots,\chi_d$ form a basis of the $\mathbb F_p$-vector space $H^1(F) \simeq H^1(\widetilde{G}_S(\mathbb Q))$. 
For each $r \in R$, the trace map 
\[
\mathrm{tr}_r : H^2(\widetilde{G}_S(\mathbb Q)) \rightarrow \mathbb F_p : \varphi \mapsto \mathrm{tg}^{-1}(\varphi)(r)
\]
is defined as an element of the dual space $H^2(\widetilde{G}_S(\mathbb Q))^{\vee}$, 
where $\mathrm{tg} : H^1(R)^{F/R} \simeq \mathrm{Hom}(R/R^p[F,R],\mathbb F_p) \stackrel{\simeq}{\longrightarrow} H^2(\widetilde{G}_S(\mathbb Q))$ 
is the transgression isomorphism. 
Then 
\[
H^2(\widetilde{G}_S(\mathbb Q))^{\vee} =\ 
\underset{i=0}{\stackrel{d}{\textstyle{\sum}}} \mathbb F_p \mathrm{tr}_{r_i}. 
\]
For $i \ge 0$ and a multi-index $I=(i_1 i_2)$ with $0 \le i_1, i_2 \le d$, 
we have 
\[
\mathrm{tr}_{r_i}(\chi_{i_1} \cup \chi_{i_2})=-\varepsilon_{I,p}(r_i)
=\delta_{p,2} \delta_{i,i_1}\delta_{i,i_2} \delta_{\ell_i+4\mathbb Z,3+4\mathbb Z}
- \delta_{i,i_1} b_{i,i_2}
+ \delta_{i,i_2} b_{i,i_1}
\]
by the basic properties of mod $p$ Magnus coefficients (cf.\ e.g.\ \cite[Theorem 2.4]{Gar14} and \cite[Proposition 2.18]{Vog05}). 

Under the assumptions, we put 
\[
U= \underset{j\,:\,\mathrm{odd}}{\textstyle{\bigoplus}} \mathbb F_p \chi_{\sigma(j)} 
\quad \hbox{and} \quad
V= \underset{j\,:\,\mathrm{even}}{\textstyle{\bigoplus}} \mathbb F_p \chi_{\sigma(j)} . 
\]
If $i_1$ and $i_2$ are even, 
we have $\mathrm{tr}_{r_i}(\chi_{\sigma(i_1)} \cup \chi_{\sigma(i_2)})=0$ 
for any $0 \le i \le d$ by the assumption. 
Hence $V \cup V=0$ by the nondegeneracy of the pairing 
$H^2(\widetilde{G}_S(\mathbb Q)) \times H^2(\widetilde{G}_S(\mathbb Q))^{\vee}  \rightarrow \mathbb F_p$. 
Put 
\[
\varphi_j = \chi_{\sigma(j)} \cup \chi_{\sigma(j+1)}=-\chi_{\sigma(j+1)} \cup \chi_{\sigma(j)} \in U \cup V, 
\]
and put 
\[
a_{i,j}=\mathrm{tr}_{r_{\sigma(i)}} \varphi_j
= -\delta_{i,j}b_{\sigma(j),\sigma(j+1)}+\delta_{i,j+1}b_{\sigma(j+1),\sigma(j)} \in \mathbb F_p
\]
for $0 \le i,j \le d$. 
Then the $(d+1) \times (d+1)$ matrix 
\[
A=(a_{i,j})_{i,j}=
\left(
\begin{array}{rrrrr}
-b_{\sigma(0),\sigma(1)} &                          &        &                            &  b_{\sigma(0),\sigma(d)} \\
 b_{\sigma(1),\sigma(0)} & -b_{\sigma(1),\sigma(2)} &        &                            &                          \\
                         &  b_{\sigma(2),\sigma(1)} & \ddots &                            &                          \\
                         &                          & \ddots & -b_{\sigma(d-1),\sigma(d)} &                          \\
                         &                          &        &  b_{\sigma(d),\sigma(d-1)} & -b_{\sigma(d),\sigma(0)} 
\end{array}
\right)
\]
has the nonzero determinant 
$\det A = \prod_{i=0}^{d} b_{\sigma(i),\sigma(i+1)} - \prod_{i=0}^{d} b_{\sigma(i+1),\sigma(i)} \neq 0$ by the assumption. 
The linearity of trace maps $\mathrm{tr}_{r_{\sigma(i)}}$ yields that 
the elements $\Phi_j \in U \cup V$ defined by 
\[
(\Phi_0,\Phi_1, \cdots, \Phi_d)=
(\varphi_0,\varphi_1, \cdots, \varphi_d)A^{-1} 
\]
satisfy $\mathrm{tr}_{r_{\sigma(i)}}\Phi_j=\delta_{i,j}$ for any $0 \le i,j \le d$. 
Hence the $d+1$ elements $\Phi_j \in U \cup V$ are linearly independent. 
This implies that $U \cup V=H^2(\widetilde{G}_S(\mathbb Q))$ 
and the deficiency 
$\widetilde{d}_S-\widetilde{r}_S=0$. 
By the cup-product criterion, 
$\widetilde{G}_S(\mathbb Q)$ is a mild pro-$p$ group. 
The latter statement also holds as the basic properties of a mild pro-$p$ group (cf.\ e.g.\ \cite[Theorem 2.4]{Gar14}). 
\end{proof}

\begin{example}
For $\ell_0=p=3$ and $(\ell_1,\ell_2,\ell_3)=(13,73,61)$, 
one can easily see that 
$\mathrm{lk}(\ell_{0},\ell_{2}) \equiv \mathrm{lk}(\ell_{2},\ell_{0}) \equiv 0 \pmod{3}$ and 
$\mathrm{lk}(\ell_{0},\ell_{1}) \mathrm{lk}(\ell_{1},\ell_{2}) \mathrm{lk}(\ell_{2},\ell_{3}) \mathrm{lk}(\ell_{3},\ell_{0}) \not\equiv 0 \equiv \mathrm{lk}(\ell_{0},\ell_{3}) \pmod{3}$. 
Then $\{p\} \cup S=\{\ell_0,\ell_1,\ell_2,\ell_3\}$ is a circular set with $\sigma$ such that $\sigma(i) \equiv i \pmod{4}$ for all $i$, 
and hence $\widetilde{G}_S(\mathbb Q)$ is a mild pro-$3$ group by Theorem \ref{thm:mild}. 
\end{example}

\begin{example}
For $\ell_0=p=2$ and $S=\{\ell_1,\ell_2,\ell_3,q\}$, 
suppose that 
$\ell_1 \equiv 7 \pmod{8}$, 
$\ell_2 \equiv 1 \pmod{8}$, 
$\ell_3 \equiv 5 \pmod{8}$, 
$\big(\frac{\ell_1}{\ell_2}\big)=\big(\frac{\ell_2}{\ell_3}\big)=-1$, 
and $q \equiv 3 \pmod{8}$. 
For example, $(\ell_1,\ell_2,\ell_3,q)=(7, 17, 5, 3)$. 
Then the assumption of Theorem \ref{thm:mild} is satisfied for $\sigma$ such that $\sigma(i) \equiv i \pmod{4}$ for all $i$, 
and hence $\widetilde{G}_S(\mathbb Q)$ is a mild pro-$2$ group. 
\end{example}

%------------------------------------------------------------------------------
\section{Alexander invariants in Iwasawa theory}\label{sec:Iwapoly}

\subsection{Iwasawa polynomial.}\label{sbsec:Iwapoly} 
Let $\varLambda=\mathbb Z_p[[T]]$ be the ring of formal power series in a variable $T$ with coefficients in $\mathbb Z_p$. 
Any finitely generated $\varLambda$-module $X$ has a finite presentation 
\[
\entrymodifiers={+!!<0pt,\fontdimen22\textfont2>}
\xymatrix{
\varLambda^{d_2} \ar[r]^-{Q} & \varLambda^{d_1} \ar[r] & X \ar[r] & 0
}
\]
with a $d_1 \times d_2$ presentation matrix $Q$ such that $d_2 \ge d_1 \ge 1$. 
Independently on the choice of such presentation, 
the $i$th elementary ideal ($i$th Fitting ideal) $E_i(X)$ is defined as an ideal of $\varLambda$ 
generated by $(d_1-i)\times(d_1-i)$ minors of $Q$ if $0 \le i < d_1$, 
and $E_i(X)=\varLambda$ if $i \ge d_1$. 
The `divisorial hull' $\widetilde{E}$ of an ideal $E$ is defined as the intersection of all principal ideals containing $E$ (cf.\ \cite{Hil}). 
Since $\varLambda$ is a unique factorization domain, $\widetilde{E}$ is a principal ideal generated by the greatest common divisor of generators of $E$ if $E \neq \{0\}$. 
Since $\mathrm{Ann}(X)^{d_1} \subset E_0(X) \subset \mathrm{Ann}(X)$ for the annihilator ideal $\mathrm{Ann}(X)$ of $X$ (cf.\ e.g.\ \cite[Appendix]{MW84}), 
we have $\widetilde{E}_0(X)=\{0\}$ (i.e., $E_0(X)=\{0\}$) 
if and only if $X$ is not $\varLambda$-torsion. 

Iwasawa polynomials are defined analogous to Alexander polynomials as follows. 
For simplicity and convenience, 
we assume that $k \cap \mathbb Q^{\mathrm{cyc}}=\mathbb Q$ and 
$k^{\mathrm{cyc}}/k$ is totally ramified at any $v \in P$. 
Let $\overline{\gamma}$ be the element of $\mathrm{Gal}(k(\mu_{p^{\infty}})/k)$ such that $\overline{\gamma}(\zeta)=\zeta^{\kappa}$ for any $\zeta \in \mu_{p^{\infty}}$, 
where $\kappa=1+p$ or $\kappa=5$ according to $p \neq 2$ or $p=2$. 
Then $\overline{\gamma}|_{k^{\mathrm{cyc}}}$ is a generator of $\mathrm{Gal}(k^{\mathrm{cyc}}/k)$, 
and there is $\widetilde{\gamma} \in \mathrm{Gal}((k^{\mathrm{cyc}})_S/k_S)$ such that $\widetilde{\gamma}|_{k^{\mathrm{cyc}}}=\overline{\gamma}|_{k^{\mathrm{cyc}}}$. 
Recall that $S \cap P =\emptyset$. 
Let $K/k$ be a finite subextension of $k_S/k$. 
Then $K \cap k^{\mathrm{cyc}}=k$, and $K^{\mathrm{cyc}}/K$ is also totally ramified at any primes lying over $p$. 
Put $\gamma = \widetilde{\gamma}|_{K^{\mathrm{cyc}}}$, which is the generator of $\varGamma=\mathrm{Gal}(K^{\mathrm{cyc}}/K)$ such that $\gamma|_{k^{\mathrm{cyc}}}=\overline{\gamma}|_{k^{\mathrm{cyc}}}$. 
The left action of $\varGamma$ on $G_S(K^{\mathrm{cyc}})$ is defined by ${^{\gamma}g}=\widetilde{\gamma}g\widetilde{\gamma}^{-1}$ for $g \in G_S(K^{\mathrm{cyc}})$. 
This induces the action of $\varLambda$ on $X=G_{\varSigma}(K^{\mathrm{cyc}})^{\mathrm{ab}}$ for any $\varSigma \subset S$ 
via the isomorphism $\varLambda \simeq \mathbb Z_p[[\varGamma]] : 1+T \leftrightarrow \gamma$, 
where $\mathbb Z_p[[\varGamma]] = \varprojlim \mathbb Z_p[\varGamma/\varGamma^{p^n}]$ is the complete group ring. 
Then the Iwasawa module $X$ is a finitely generated $\varLambda$-torsion $\varLambda$-module. 
The `Iwasawa polynomial' $\varDelta(T)=p^{\mu}P(T) \in \mathbb Z_p[T]$ of $X$ is defined as the generator of 
the divisorial hull $\widetilde{E}_0(X)=\varDelta(T)\varLambda$ of $E_0(X)$ 
(cf.\ Remark \ref{rem:dvhll} below) 
with $0 \le \mu \in \mathbb Z$ and monic $P(T) \equiv T^{\lambda} \pmod{p}$, 
where $\lambda=\deg P(T)$. 
There exists $\nu \in \mathbb Z$ satisfying Iwasawa's class number formula 
$|G_{\varSigma}(K_n)^{\mathrm{ab}}|=p^{\lambda n + \mu p^n + \nu}$ 
for all sufficiently large $n$ (cf.\ e.g.\ \cite{Was}), 
where $K_n$ is the fixed field of $\varGamma^{p^n}$. 
The theorem of Ferrero-Washington \cite{FW79} implies that 
$\mu=0$ if $K/\mathbb Q$ is an abelian extension. 

\begin{remark}\label{rem:dvhll}
Suppose that there is a $\varLambda$-homomorphism $f : X \rightarrow Y$ of finitely generated $\varLambda$-modules with finite cokernel. 
Then $(p^n,T^n) \subset E_0(\mathrm{Coker}f)$ for sufficiently large $n$. 
By the basic properties of Fitting ideals (cf.\ e.g.\ \cite[Appendix]{MW84}), 
we have 
\[
(p^n,T^n) E_0(X) \subset (p^n,T^n) E_0(\mathrm{Im}f) \subset E_0(\mathrm{Coker}f)E_0(\mathrm{Im}f) \subset E_0(Y) , 
\]
and hence $\widetilde{E}_0(X) \subset \widetilde{E}_0(Y)$. 
Moreover if $\mathrm{Ker}f$ is also finite (i.e., $f$ is a pseudo-isomorphism), 
and if $X$ and $Y$ are $\varLambda$-torsion, 
then there is a pseudo-isomorphism $g : Y \rightarrow X$, 
and hence $\widetilde{E}_0(X)=\widetilde{E}_0(Y)$. 
In particular when $Y=\bigoplus_{i=1}^{m} \varLambda/\wp_i^{n_i}$ 
with some prime ideals $\wp_i$ of height $1$ and $n_i \ge 1$, 
we have $\widetilde{E}_0(X)=\prod_{i=1}^{m} \wp_i^{n_i}$. 
\end{remark}

Based on the analogy between $\widetilde{G}_{\emptyset}(k)$ and $\pi_1(\mathsf{X})$, 
we obtain the following another proof of Gold's theorem 
analogous to \cite[Theorem 3.2]{KM13}. 

\begin{theorem}[cf.\ \cite{Gol74}, also \cite{San93}]\label{thm:Gold}
Assume that $p \neq 2$ and $k$ is an imaginary quadratic field with class number $h_k \not\equiv 0 \pmod{p}$, 
and that $p$ splits in $k$ as $pO_k=\mathfrak{p}_1\mathfrak{p}_2$. 
Let $\varDelta(T)$ be the Iwasawa polynomial of $X=G_{\emptyset}(k^{\mathrm{cyc}})^{\mathrm{ab}}$. 
Then $\varDelta(T)=T$ (i.e., $\widetilde{G}_{\emptyset}(k) \simeq \mathbb Z_p^2$) 
if and only if 
$\mathrm{lk}(\mathfrak{p}_1,\mathfrak{p}_2) \not\equiv 0 \pmod{p}$. 
\end{theorem}

\begin{proof}
Suppose $K=k$ and $\varSigma=\emptyset$. 
By Theorem \ref{thm:presk}, $\widetilde{G}_{\emptyset}(k)$ has a presentation 
\[
\entrymodifiers={+!!<0pt,\fontdimen22\textfont2>}
\xymatrix{
1 \ar[r] & R \ar[r] & F \ar[r]^-{\pi} & \widetilde{G}_{\emptyset}(k) \ar[r] & 1 
}
\]
with a free pro-$p$ group $F=\langle x_1,x_2 \rangle$ and the normal subgroup $R$ which is normally generated by $r_1, r_2$ of the form 
\[
r_1 \equiv r_2^{-1} \equiv [x_1^{-1},x_2^{-1}]^{\mathrm{lk}(\mathfrak{p}_1,\mathfrak{p}_2)} \mod{F_3} 
\]
(cf.\ also Remark \ref{rem:lkimag}). 
Then $RF_3=F_2^{\mathrm{lk}(\mathfrak{p}_1,\mathfrak{p}_2)}F_3$. 
In particular $R \subset F_2$, 
i.e., $\widetilde{G}_{\emptyset}(k)^{\mathrm{ab}} \simeq \mathbb Z_p^2$, 
and there is a surjective homomorphism $X \twoheadrightarrow \varLambda/T\varLambda$. 
Hence $\varDelta(T) \in T\varLambda$. 
It is well known that $X \simeq \mathbb Z_p^{\lambda}$ 
(cf.\ e.g.\ \cite[Corollary 13.29]{Was}). 
Therefore $\widetilde{G}_{\emptyset}(k) \simeq \mathbb Z_p^2$ if and only if $X \simeq \varLambda/T\varLambda$ (i.e., $\varDelta(T)=T$). 
Since $F_2=R$ if and only if $F_2=F_3R$, 
we obtain the claim. 
\end{proof}

\begin{remark}\label{rem:anlgA}
An analogue of Iwasawa module is a finitely generated $\mathbb Z[t^{\pm 1}]$-module $H_1(\mathsf{X}^{\mathrm{cyc}},\mathbb Z) \simeq \pi_1(\mathsf{X}^{\mathrm{cyc}})^{\mathrm{ab}}$, 
where $\mathsf{X}^{\mathrm{cyc}}$ is an infinite cyclic covering of $\mathsf{X}$ with $\Gamma=\mathrm{Gal}(\mathsf{X}^{\mathrm{cyc}}/\mathsf{X})=t^{\mathbb Z}$. 
Then the Alexander polynomial $\Delta(t) \in \mathbb Z[t^{\pm 1}]$ is defined as a generator of $\widetilde{E}_0(H_1(\mathsf{X}^{\mathrm{cyc}},\mathbb Z))=\Delta(t)\mathbb Z[t^{\pm 1}]$. 
Regarding $\Delta(1+T) \in \varLambda$, 
Iwasawa invariants $\lambda$, $\mu$ and $\nu$ for $\mathsf{X}^{\mathrm{cyc}}/\mathsf{X}$ are analogously defined and studied (cf.\ \cite{HMM,KM08,KM13,KM14,Mor02,Uek16,Uek} etc.). 
For $\mathsf{L}=\mathsf{K}_1 \cup \mathsf{K}_2 \subset \mathsf{M}=S^3$ (and $\mathsf{X}^{\mathrm{cyc}}$ such that the meridians of $\mathsf{K}_1$ and $\mathsf{K}_2$ are nontrivial in $\Gamma$), 
an analogue \cite[Theorem 3.2]{KM13} of Gold's theorem states that 
$\Delta(1+T)\varLambda=\varLambda T$ if and only if the linking number $\mathrm{lk}(\mathsf{K}_1,\mathsf{K}_2) \not\equiv 0 \pmod{p}$. 
\end{remark}

Alexander polynomials can be calculated from a presentation of $\pi_1(\mathsf{X})$ by Fox derivative (cf.\ \cite{Fox53,Mor} etc.). 
Analogously, in some special cases, 
one can calculate approximation of Iwasawa polynomials 
by pro-$p$ Fox derivative 
as in the following sections.

\subsection{Iwasawa module as a subquotient.} 
Suppose that $p=2$, $k=\mathbb Q$, $S^{\ast}=\{\ell_1,\cdots,\ell_d \} \subset S$ with $\ell_i \equiv 1 \pmod{2}$, 
and $\varSigma=S\setminus S^{\ast}$ is either 
$\{\infty\}$ or $\{q\}$ with $q \equiv 3 \pmod{4}$. 
Let $K=\mathbb Q(\sqrt{D}) \subset \mathbb Q_S$ be the quadratic field of discriminant 
\[
D=
\left\{
\begin{array}{ll}
\prod_{i=1}^d \ell_i^* & \hbox{if } \varSigma=\{\infty\}, \\
\prod_{i=1}^d \ell_i \quad \hbox{or} \quad q\prod_{i=1}^d \ell_i & \hbox{if } \varSigma=\{q\}, 
\end{array}
\right.
\]
where $\ell_i^* = (-1)^{\frac{\ell_i-1}{2}} \ell_i$. 
Let 
\[
\entrymodifiers={+!!<0pt,\fontdimen22\textfont2>}
\xymatrix{
1 \ar[r] & R \ar[r] & F \ar[r]^-{\pi} & \widetilde{G}_S(\mathbb Q) \ar[r] & 1
}
\]
be the presentation of $\widetilde{G}_S(\mathbb Q)$ obtained in Theorem \ref{thm:presQ}. 
Put the closed subgroup
\[
H=\langle x_0, x_d^{-1}x_0x_d, x_1x_d,\cdots,x_{d-1}x_d, x_1^2,\cdots,x_d^2 \rangle
\]
of $F$ generated by such $2d+1$ elements, where we ignore $x_1x_d,\cdots,x_{d-1}x_d$ if $d=1$. 
Recall that $\pi(x_i)$ generates the inertia group $T_i$ of $\widetilde{\ell}_i$ in $\widetilde{G}_S(\mathbb Q)$. 
Then $\mathrm{Gal}(K/\mathbb Q)=\langle \pi(x_i)|_K \rangle$ for any $1 \le i \le d$, and $\pi(x_0)|_K=1$. 
Hence $H$ is contained in the kernel of the surjective homomorphism
\[
\entrymodifiers={+!!<0pt,\fontdimen22\textfont2>}
\xymatrix{
|_K \circ \pi : F \ar[r]^-{\pi} & \widetilde{G}_S(\mathbb Q) \ar[r]^-{|_K} & \mathrm{Gal}(K/\mathbb Q) 
}.
\]
Since $x_i^{-1}x_0x_i=(x_i^{-1}x_d)(x_d^{-1}x_0x_d)(x_d^{-1}x_i)$ and 
\[
(x_ix_j^{-1})x_j^2=x_i^2(x_i^{-1}x_j)=x_ix_j=(x_ix_d)(x_jx_d)^{-1}x_j^2 \in H
\]
for $1 \le i,j \le d$, 
one can easily see that $H$ is a normal subgroup of $F$ such that 
$F/H = \langle x_iH \rangle \simeq \mathbb Z/2\mathbb Z$ for any $1 \le i \le d$. 
Then the maximal subgroup $H$ of $F$ is a free pro-$2$ group freely generated by the $2d+1$ elements (cf.\ e.g.\ \cite[Corollary (3.9.6)]{NSW}), 
and $R \subset F^2F_2 \subset H$. 
Since $|_K \circ \pi$ induces the isomorphism 
$(F/R)/(H/R) \simeq \mathrm{Gal}(K/\mathbb Q)$, 
$\pi$ induces a presentation 
\[
\entrymodifiers={+!!<0pt,\fontdimen22\textfont2>}
\xymatrix{
1 \ar[r] & R \ar[r] & H \ar[r]^-{\pi} & \widetilde{G}_S(K) \ar[r] & 1
}
\]
of $\widetilde{G}_S(K)=\mathrm{Gal}((\mathbb Q^{\textrm{cyc}})_S/K)$. 
The inertia group 
\[
T_i \cap \widetilde{G}_S(K) = \mathrm{Ker}(T_i \rightarrow \mathrm{Gal}(K/\mathbb Q) : \pi(x_i) \mapsto \pi(x_i)|_K)
\]
of $\widetilde{\ell}_i$ in $\widetilde{G}_S(K)$ is generated by $\pi(x_i^2)$. 
Hence $\pi$ also induces a presentation 
\[
\entrymodifiers={+!!<0pt,\fontdimen22\textfont2>}
\xymatrix{
1 \ar[r] & NR \ar[r] & H \ar[r]^-{\varpi} & \widetilde{G}_{\varSigma}(K) \ar[r] & 1
}
\]
of $\widetilde{G}_{\varSigma}(K)=\mathrm{Gal}((K^{\textrm{cyc}})_{\varSigma}/K)$, 
where $\varpi=|_{(K^{\textrm{cyc}})_{\varSigma}} \circ \pi$, 
and 
$N=\langle x_1^2,\cdots,x_d^2 \rangle_H$
is the normal subgroup of $H$ 
normally generated by $x_1^2,\cdots,x_d^2$. 
Note that $g^{-1}x_i^2g=(x_i^{-1}g)^{-1}x_i^2(x_i^{-1}g) \in N$ for any $g \in F=H \cup x_iH$ and $1 \le i \le d$. 
Then $N$ is also a normal subgroup of $F$, and hence a presentation 
\[
\entrymodifiers={+!!<0pt,\fontdimen22\textfont2>}
\xymatrix{
1 \ar[r] & NR \ar[r] & F \ar[r]^-{\varpi} & \mathrm{Gal}((K^{\textrm{cyc}})_{\varSigma}/\mathbb Q) \ar[r] & 1
}
\]
of $\mathrm{Gal}((K^{\textrm{cyc}})_{\varSigma}/\mathbb Q)$ is also induced. 
Recall that $R$ is generated by $r_0,r_1,\cdots,r_d$ as a normal subgroup of $F$, 
and put 
\[
\rho_i = [x_i^{-1},y_i^{-1}] \equiv r_i \mod{N}
\]
for $0 \le i \le d$. 
Since 
\[
g^{-1}\rho_ig = (x_i^{-1}g)^{-1}(y_ix_i^2y_i^{-1})^{-1} \rho_i^{-1} x_i^2 (x_i^{-1}g) 
\equiv (x_i^{-1}g)^{-1} \rho_i^{-1} (x_i^{-1}g) \mod{N}
\]
for any $g \in F=H \cup x_iH$ and $1 \le i \le d$, 
\[
NR=\langle x_1^2,\cdots,x_d^2, \rho_0,x_d^{-1}\rho_0x_d,\rho_1,\cdots,\rho_d \rangle_H
\]
is normally generated by such $2d+2$ elements in $H$. 
Put $\overline{F}=F/N$, and put $\overline{g}=gN \in \overline{F}$ for any $g \in F$. 
The universality of the free pro-$2$ group implies that 
\[
\overline{H}=H/N=\langle w_{01},w_{02},w_1,\cdots,w_{d-1} \rangle
\]
is a free pro-$2$ group of rank $d+1$, where 
\[
w_{01}=\overline{x_0}, \quad
w_{02}=\overline{x_d^{-1}x_0x_d}, \quad
w_i=\overline{x_ix_d} 
\]
for $1 \le i \le d-1$. 
We ignore $w_i$ if $d=1$. 
Then $\varpi$ induces a presentation
\[
\entrymodifiers={+!!<0pt,\fontdimen22\textfont2>}
\xymatrix{
1 \ar[r] & \overline{R} \ar[r] & \overline{H} \ar[r]^-{\overline{\varpi}} & \widetilde{G}_{\varSigma}(K) \ar[r] & 1
}
\]
of $\widetilde{G}_{\varSigma}(K)$ and an exact sequence 
\[
\entrymodifiers={+!!<0pt,\fontdimen22\textfont2>}
\xymatrix{
1 \ar[r] & \overline{R} \ar[r] & \overline{F} \ar[r]^-{\overline{\varpi}} & \mathrm{Gal}((K^{\textrm{cyc}})_{\varSigma}/\mathbb Q) \ar[r] & 1
},
\]
where 
\[
\overline{R}=NR/N=\langle \overline{\rho_0},\overline{x_d^{-1}\rho_0x_d},\overline{\rho_1},\cdots,\overline{\rho_d} \rangle_{\overline{H}}
\]
and $\overline{\varpi}(\overline{g})=\varpi(g)=\pi(g)|_{(K^{\textrm{cyc}})_{\varSigma}}$ for $g \in F$. 
Recall that we can put $\widetilde{\gamma}=\pi(x_0)$ (cf.\ Remark \ref{rem:gamma}). 
Then $\varGamma=\mathrm{Gal}(K^{\textrm{cyc}}/K)=\langle \gamma \rangle$ 
with $\gamma=\widetilde{\gamma}|_{K^{\textrm{cyc}}}$. 
Let $\tau$ be the generator of $\mathrm{Gal}(K^{\mathrm{cyc}}/\mathbb Q^{\mathrm{cyc}}) \simeq \mathbb Z/2\mathbb Z$. 
Then $\mathrm{Gal}(K^{\mathrm{cyc}}/\mathbb Q)=\langle \tau, \gamma \rangle \simeq \mathbb Z/2\mathbb Z \oplus \mathbb Z_p$, 
and $\pi(x_i)|_{K^{\mathrm{cyc}}}=\tau$ for all $1 \le i \le d$. 
Since $\overline{\varpi}(w_{01})|_{K^{\mathrm{cyc}}}=\overline{\varpi}(w_{02})|_{K^{\mathrm{cyc}}}=\gamma$ and $\overline{\varpi}(w_i)|_{K^{\mathrm{cyc}}}=1$, 
the kernel of 
\[
\entrymodifiers={+!!<0pt,\fontdimen22\textfont2>}
\xymatrix{
|_{K^{\mathrm{cyc}}} \circ \overline{\varpi} : \overline{H} \ar[r]^-{\overline{\varpi}} & \widetilde{G}_{\varSigma}(K) \ar[r]^-{|_{K^{\mathrm{cyc}}}} & \varGamma 
}
\]
is 
\[
\overline{M}=\langle w_{01}w_{02}^{-1},w_1,\cdots,w_{d-1} \rangle_{\overline{H}}. 
\]
Then $\overline{H}/\overline{M}=\langle w_{01}\overline{M} \rangle \simeq \mathbb Z_2$ and $\overline{\varpi}(\overline{M})=G_{\varSigma}(K^{\mathrm{cyc}})$. 
By defining the action of $\gamma$ on $m \in \overline{M}$ as 
\[
{^{\gamma}m}=w_{01} m w_{01}^{-1} , 
\] 
$\overline{M}$ is a free pro-$2$-$\varGamma$ operator group of rank $d$, 
and $\overline{M}/\overline{M}_2 \simeq \varLambda^d$ is a free $\varLambda$-module of rank $d$. 
Since $\overline{F}/\overline{H}=\langle \overline{x_d} \overline{H} \rangle$ and 
\begin{align*}
\overline{x_d}(w_{01}w_{02}^{-1})\overline{x_d^{-1}} &= (w_{01}w_{02}^{-1})^{-1} = [\overline{x_d},w_{01}^{-1}], \\
\overline{x_d} w_i \overline{x_d^{-1}} &= w_i^{-1}
\end{align*}
for all $1 \le i \le d-1$, 
$\overline{M}$ is a normal subgroup of $\overline{F}$, 
and $\overline{\varpi}$ induces the isomorphism 
\[
\overline{F}/\overline{M}=\langle \overline{x_d} \overline{M}, w_{01}\overline{M} \rangle \simeq \mathrm{Gal}(K^{\mathrm{cyc}}/\mathbb Q)=\langle \tau, \gamma \rangle : \overline{x_d} \overline{M} \mapsto \tau, w_{01}\overline{M} \mapsto \gamma .
\]
In particular, we have $\overline{R} \subset \overline{F}_2 \subset \overline{M}$. 
Note that 
\[
\overline{x_d}(\overline{h} m \overline{h^{-1}})\overline{x_d^{-1}}
=[\overline{x_d^{-1}},\overline{h^{-1}}]
\overline{h} (\overline{x_d} m \overline{x_d^{-1}}) \overline{h^{-1}}
[\overline{h^{-1}},\overline{x_d^{-1}}]
\equiv \overline{h} (\overline{x_d} m \overline{x_d^{-1}}) \overline{h^{-1}}
\mod{\overline{M}_2} 
\]
for any $m \in \overline{M}$ and $h \in H$. 
By defining the action of $\tau$ on $\overline{M}/\overline{M}_2$ as 
\begin{align}\label{approx:tau}
{^{\tau}(m\overline{M}_2)}=\overline{x_d} m \overline{x_d^{-1}} \overline{M}_2 = (m\overline{M}_2)^{-1}
\end{align}
for $m \in \overline{M}$, 
we regard $\overline{M}/\overline{M}_2$ as a $\mathbb Z_2[[\mathrm{Gal}(K^{\mathrm{cyc}}/\mathbb Q)]]$-module, 
particularly as a $\varLambda$-module. 
Since $\overline{M}_2\overline{R}/\overline{M}_2$ is a $\mathbb Z_2[[\mathrm{Gal}(K^{\mathrm{cyc}}/\mathbb Q)]]$-submodule of $\overline{M}/\overline{M}_2$, 
$\overline{\varpi}$ induces an isomorphism 
\[
\overline{M}/\overline{M}_2\overline{R} \simeq G_{\varSigma}(K^{\mathrm{cyc}})^{\mathrm{ab}}
\]
as $\mathbb Z_2[[\mathrm{Gal}(K^{\mathrm{cyc}}/\mathbb Q)]]$-modules, 
particularly as $\varLambda$-modules. 
Since 
\[
\overline{x_d^{-1}\rho_0x_d}\,\overline{M}_2 = {^{\tau}(\overline{\rho_0}\overline{M}_2)} =(\overline{\rho_0}\overline{M}_2)^{-1}, 
\]
this isomorphism induces a presentation 
\[
\entrymodifiers={+!!<0pt,\fontdimen22\textfont2>}
\xymatrix{
\varLambda^{d+1} \ar[r] \ar@{>>}[d] & \varLambda^d \ar[r] & G_{\varSigma}(K^{\mathrm{cyc}})^{\mathrm{ab}} \ar[r] & 0 \\
\overline{M}_2\overline{R}/\overline{M}_2 \ar[r] & \overline{M}/\overline{M}_2 \ar[u]^-{\simeq} & & \\
}
\]
of the Iwasawa module $X=G_{\varSigma}(K^{\mathrm{cyc}})^{\mathrm{ab}}$, 
which implies that the Iwasawa polynomial is computable in principle from the presentation of $\widetilde{G}_S(\mathbb Q)$. 
We obtain the following theorem concerning an approximate computation of the initial Fitting ideal. 

\begin{theorem}\label{thm:approx}
Under the settings above, 
let $E_0(X)$ be the initial Fitting ideal of the $\varLambda$-module $X=G_{\varSigma}(K^{\mathrm{cyc}})^{\mathrm{ab}}$, 
and put 
\[
\varepsilon=
\left\{
\begin{array}{ll}
1 & \hbox{if } D \equiv 1 \pmod{8}, \\
0 & \hbox{if } D \equiv 5 \pmod{8}. 
\end{array}
\right.
\]
Suppose that 
\[
\overline{\rho_i}=\overline{[x_i^{-1},y_i^{-1}]} \equiv \overline{\rho_{i,n}} 
\mod{\mathfrak{M}^{(n)}}
\]
with some $\overline{\rho_{i,n}} \in \overline{H}$ for each $0 \le i \le d$ and $n \ge 3$, 
where 
\[
\mathfrak{M}^{(n)}={\textstyle \overline{M}^{2^{n-1-\varepsilon}} \overline{M}_2 \underset{i=2}{\stackrel{n}{\prod}} \overline{F}_{i}^{2^{n-i}}}. 
\]
Then the ideal $E_0(X)+(2,T)^{n-1-\varepsilon}$ of $\varLambda$ is generated by 
$(2,T)^{n-1-\varepsilon}$ 
and the $d \times d$ minors of the $d \times (d+1)$ matrix 
\[
Q_n=\left(
\varPhi
\Big(\frac{\partial \overline{\rho_{i,n}}}{\partial w}\Big)
\right)_{w,i}
\]
with rows and columns indexed by $w \in \{w_{01},w_1,\cdots,w_{d-1}\}$ and $0 \le i \le d$, 
where 
\[
\entrymodifiers={+!!<0pt,\fontdimen22\textfont2>}
\xymatrix{
\varPhi : \mathbb Z_2[[\overline{H}]] \ar[r] & \mathbb Z_2[[\varGamma]] \ar[r]^-{\simeq} & \varLambda
}
\]
is the $\mathbb Z_2$-linear ring homomorphism naturally extended from $|_{K^{\mathrm{cyc}}} \circ \overline{\varpi} : \overline{H} \rightarrow \varGamma$. 
\end{theorem}

\subsection{Proof of Theorem \ref{thm:approx}.} 
Put $\overline{M}^{(1)}=\overline{M} \supset \overline{F}_2$, and put $\overline{M}^{(n)}=[\overline{M}^{(n-1)},\overline{F}] \supset \overline{F}_{n+1}$ for $n \ge 2$ recursively. 
Arbitrary $\overline{g} \in \overline{F}$ can be written in the form 
$\overline{g}=w_{01}^z \overline{x_d}^e m'$ 
with some $z \in \mathbb Z_2$, $e \in \{0,1\}$ and $m' \in \overline{M}$. 
Then, since 
\[
[m,\overline{g}] \equiv m^{-2e}(({^{\gamma^{-z}} m})m^{-1})^{(-1)^e} \mod{\overline{M}_2} 
\]
for any $m \in \overline{M}$, 
we have 
\begin{align}\label{approx:M}
\overline{M}^{(n)}\overline{M}_2/\overline{M}_2
= (2,T)^{n-1} (\overline{M}/\overline{M}_2) 
\end{align}
for any $n \ge 1$ by induction. 
Then $\overline{F}_{i}^{2^{n-i}} \subset (\overline{M}^{(i-1)})^{2^{n-i}} \subset \overline{M}^{(n-1)}\overline{M}_2$ for any $n \ge i \ge 2$, 
and hence 
\begin{align}\label{approx:modin}
\mathfrak{M}^{(n)} \subset \overline{M}^{(n-1)}\overline{M}_2 \subset \overline{M}^{(2)}\overline{M}_2
\end{align}
for $n \ge 3$. 
In particular, we have  
\[
\overline{M}^{(2)}\overline{M}_2\overline{R}=\overline{M}^{(2)}\overline{M}_2\overline{R}^{(n)}, 
\]
where we put 
\[
\overline{R}^{(n)}=\langle \overline{\rho_{0,n}},\overline{x_d^{-1}\rho_{0,n}x_d},\overline{\rho_{1,n}},\cdots,\overline{\rho_{d,n}} \rangle_{\overline{H}} 
\subset \overline{M}. 
\]

\begin{lemma}\label{lem:approx}
$\overline{M}^{(n-\varepsilon)}\overline{M}_2\overline{R} 
= \overline{M}^{(n-\varepsilon)}\overline{M}_2\overline{R}^{(n)}$ 
for any $n \ge 3$. 
\end{lemma}

\begin{proof}
If $D \equiv 1 \pmod{4}$, we obtain the claim by (\ref{approx:modin}). 
Suppose that $D \equiv 5 \pmod{8}$. 
Let $L$ be the fixed field of $(2,T)X$. 
Then $L/K^{\mathrm{cyc}}$ is the maximal elementary abelian $2$-extension 
unramified outside $\varSigma$ which is abelian over $K$. 
Since 
\[
\overline{M}/\overline{M}^{(2)}\overline{M}_2\overline{R}
\simeq X/(2,T)X \simeq \mathrm{Gal}(L/K^{\mathrm{cyc}})
\]
as $\mathbb Z_2[[\mathrm{Gal}(K^{\mathrm{cyc}})/\mathbb Q]]$-modules by (\ref{approx:M}), 
we have 
\[
\overline{F}/\overline{M}^{(2)}\overline{M}_2\overline{R}
\simeq \mathrm{Gal}(L/\mathbb Q). 
\]
Let $L'$ be the inertia field in $L/K$ of the unique prime of $K$ lying over $2$. 
Since $L/L'$ and $K^{\mathrm{cyc}}/K$ are totally ramified, 
we have $L=L'K^{\mathrm{cyc}}$ and $L' \cap K^{\mathrm{cyc}} =K$. 
Hence $L'$ is an elementary abelian $2$-extension of $K$ unramified outside $\varSigma$. 
Since $G_{\varSigma}(\mathbb Q)^{\mathrm{ab}} \simeq \{1\}$, 
$1+\tau$ annihilates the Sylow $2$-subgroup of the ray class group of $K$ modulo $v \in \varSigma$, 
i.e., 
$\tau|_K \in \mathrm{Gal}(K/\mathbb Q)$ acts as inverse on $G_{\varSigma}(K)^{\mathrm{ab}}$. 
The trivial action of $\mathrm{Gal}(K/\mathbb Q)$ on $G_{\varSigma}(K)^{\mathrm{ab}}/2$ implies that 
the maximal elementary abelian $2$-extension of $K$ unramified outside $\varSigma$ is an abelian extension over $\mathbb Q$. 
Therefore $L=L' \mathbb Q^{\mathrm{cyc}}$ is abelian over $\mathbb Q$, i.e., 
\[
\overline{F}_2 \subset \overline{M}^{(2)}\overline{M}_2\overline{R}=\overline{M}^{(2)}\overline{M}_2\overline{R}^{(n)}. 
\]
Since $\overline{M}_2\overline{R}^{(n)}$ is also a normal subgroup of $\overline{F}$ by (\ref{approx:tau}), 
we have 
\[
\overline{F}_i \subset \overline{M}^{(i)}\overline{M}_2\overline{R} \cap \overline{M}^{(i)}\overline{M}_2\overline{R}^{(n)}
\]
for any $i \ge 2$. 
Since 
$(\overline{M}^{(i)})^{2^{n-i}} \subset \overline{M}^{(n)}\overline{M}_2$ for any $i \ge 1$ by (\ref{approx:M}), 
we have 
\[
\mathfrak{M}^{(n)} \subset \overline{M}^{(n)}\overline{M}_2\overline{R} \cap \overline{M}^{(n)}\overline{M}_2\overline{R}^{(n)}, 
\]
and hence 
\begin{align*}
\overline{M}^{(n)}\overline{M}_2\overline{R} 
= \overline{M}^{(n)}\overline{M}_2 (\mathfrak{M}^{(n)} \overline{R})
= \overline{M}^{(n)}\overline{M}_2 (\mathfrak{M}^{(n)} \overline{R}^{(n)})
= \overline{M}^{(n)}\overline{M}_2\overline{R}^{(n)}. 
\end{align*}
Thus Lemma \ref{lem:approx} is proved. 
\end{proof}

Lemma \ref{lem:approx} above and (\ref{approx:M}) yield that 
\begin{align}\label{approx:X}
X/(2,T)^{n-1-\varepsilon}X \simeq \overline{M}/\overline{M}^{(n-\varepsilon)}\overline{M}_2\overline{R} \simeq X^{(n)}/(2,T)^{n-1-\varepsilon}X^{(n)} , 
\end{align}
where 
\[
X^{(n)}=\overline{M}/\overline{M}_2\overline{R}^{(n)} \simeq 
\mathrm{Ker}(\overline{H}/\overline{R}^{(n)} \stackrel{\psi_n}{\longrightarrow} \varGamma)^{\mathrm{ab}}. 
\]
Since $\pi_n : \overline{H} \rightarrow \overline{H}/\overline{R}^{(n)}$ satisfies $\psi_n \circ \pi_n = |_{K^{\mathrm{cyc}}} \circ \overline{\varpi}$, 
there is a presentation 
\[
\entrymodifiers={+!!<0pt,\fontdimen22\textfont2>}
\xymatrix{
\varLambda^{d+2} \ar[r]^-{\widetilde{Q}_n} & \varLambda^{d+1} \ar[r] & \mathfrak{A}_{\psi_n} \ar[r] & 0 
}
\]
of the complete $\psi_n$-differential module $\mathfrak{A}_{\psi_n}$ (cf.\ \cite[Corollary 9.15]{Mor}) 
with the $(d+1) \times (d+2)$ matrix $\widetilde{Q}_n$ 
obtained by adding to $Q_n$ a row with $w=w_{02}$ and 
a column 
\[
\left(
\varPhi
\Big(\frac{\partial \overline{x_d^{-1}\rho_{0,n}x_d}}{\partial w}\Big)
\right)_{w}
=\left(
-\varPhi
\Big(\frac{\partial \overline{\rho_{0,n}}}{\partial w}\Big)
\right)_{w} 
\]
which is the inverse of the column of $i=0$
by (\ref{approx:tau}) and the following lemma. 

\begin{lemma}\label{lem:linearFox}
For any $w \in \{w_{01},w_{02},w_1,\cdots,w_{d-1}\}$, 
we have 
\[
\varPhi
\Big(\frac{\partial m_1^{z_1} m_2^{z_2}}{\partial w}\Big)
=
z_1\varPhi
\Big(\frac{\partial m_1}{\partial w}\Big)
+
z_2\varPhi
\Big(\frac{\partial m_2}{\partial w}\Big)
\]
for any $z_1,z_2 \in \mathbb Z_2$ and $m_1,m_2 \in \overline{M}$. 
In particular, 
\[
\varPhi \Big(\frac{\partial m}{\partial w}\Big)
=0
\]
if $m \in \overline{M}_2$. 
\end{lemma}

\begin{proof}
Recall that $\overline{M}$ is the kernel of $|_{K^{\mathrm{cyc}}} \circ \overline{\varpi}$, i.e., $\overline{\varpi}(m_i)|_{K^{\mathrm{cyc}}}=1$. 
Hence the claim holds 
by the basic properties (cf.\ e.g.\ \cite[Proposition 8.13]{Mor}) of the pro-$p$ Fox derivative $\frac{\partial}{\partial w}$ and the continuity of $\varPhi \circ \frac{\partial}{\partial w}$. 
\end{proof}

The involution $\tau$ on $\overline{H}$ defined by $\overline{h}^{\tau}=\overline{x_d^{-1}} \overline{h} \overline{x_d}$ satisfies 
$(w_{01})^{\tau}=w_{02}$ and $(w_i)^{\tau}=w_i^{-1}$, 
in particular, 
$\overline{\varpi}(\overline{h}^{\tau})|_{K^{\mathrm{cyc}}}=\overline{\varpi}(\overline{h})|_{K^{\mathrm{cyc}}}$ for any $\overline{h} \in \overline{H}$. 
By the chain rule of the Fox derivative (cf.\ \cite[(2.6)]{Fox53}), 
we have 
\[
\frac{\partial \overline{\rho_{i,n}}}{\partial w_{02}}
=\frac{\partial (\overline{\rho_{i,n}}^{\tau})^{\tau}}{\partial w_{02}}
=\sum_{w} 
\Big(\frac{\partial \overline{\rho_{i,n}}^{\tau}}{\partial w}\Big)^{\!\tau}
\frac{\partial w^{\tau}}{\partial w_{02}}
=
\Big(\frac{\partial \overline{\rho_{i,n}}^{\tau}}{\partial w_{01}}\Big)^{\!\tau} ,
\]
and hence (\ref{approx:tau}) and Lemma \ref{lem:linearFox} yield that 
\[
\varPhi \Big(\frac{\partial \overline{\rho_{i,n}}}{\partial w_{02}}\Big)
=
-\varPhi \Big(\frac{\partial \overline{\rho_{i,n}}}{\partial w_{01}}\Big)
\]
for any $0 \le i \le d$, 
i.e., the row of $w=w_{02}$ 
is the inverse of the row of $w=w_{01}$ in $\widetilde{Q}_n$. 
Since the Crowell exact sequence (cf.\ \cite{NQD} or \cite[Theorem 9.17]{Mor})
\[
\entrymodifiers={+!!<0pt,\fontdimen22\textfont2>}
\xymatrix{
0 \ar[r] & X^{(n)} \ar[r] & \mathfrak{A}_{\psi_n} \ar[r] & \varLambda \ar[r]^-{\varepsilon_{\mathbb Z_2[[\varGamma]]}} & \mathbb Z_2 \ar[r] & 0
}
\]
yields that $\mathfrak{A}_{\psi_n} \simeq X^{(n)} \oplus \varLambda$, 
the Fitting ideal $E_0(X^{(n)})=E_1(\mathfrak{A}_{\psi_n})$ of the $\varLambda$-module $X^{(n)}$ is generated by $d \times d$ minors of $\widetilde{Q}_n$, i.e., of $Q_n$ (cf.\ \cite[Example 9.18]{Mor}). 
Since 
\[
E_0(X)+(2,T)^{n-1-\varepsilon}=E_0(X^{(n)})+(2,T)^{n-1-\varepsilon}
\]
by (\ref{approx:X}) (cf.\ \cite[Appendix]{MW84}), we obtain the claim of Theorem \ref{thm:approx}.

\subsection{Application of Theorem \ref{thm:approx}.} 
In the following, we assume that $\varSigma=\{\infty\}$. 
Recall that $\rho_i=[x_i^{-1},y_i^{-1}]$. 
By Theorem \ref{thm:presQ} and (\ref{def:ciab}), 
\[
y_i \equiv x_d^{c_{id}} \cdots x_1^{c_{i1}} x_0^{c_{i0}}
\underset{a < b}{\textstyle{\prod}} [x_a,x_b]^{c_{iab}} \mod{F_3N}
\]
with $c_{i1},\cdots,c_{id} \in \{0,1\}$ and $c_{iab} \in \mathbb Z_2$ 
for $0 \le i \le d$, 
where 
\begin{align*}
c_{i0} &= \mathrm{lk}(\ell_i,\ell_0), \\
c_{ij} &\equiv \mathrm{lk}(\ell_i,\ell_j) \mod{2} 
\end{align*}
for $1 \le j \le d$. 
Then we have 
\begin{align}\label{eqv:rho}
\rho_i &\equiv \ 
\underset{j=0}{\stackrel{d}{\textstyle{\prod}}} [x_i^{-1},x_j^{-1}]^{c_{ij}} 
\cdot [x_i,x_0,x_0]^{c_{i0}(c_{i0}-1)/2} \cdot 
\underset{a < b}{\textstyle{\prod}} 
[x_i,x_a,x_b]^{c_{ia}c_{ib}} [x_a,x_b,x_i]^{c_{iab}} 
\\
& \bmod{F_3^2 F_4(F_2)_2N} \nonumber
\end{align}
by \cite[Propositions (3.8.3) and (3.8.6)]{NSW} and the multilinearity of brackets 
\[
\entrymodifiers={+!!<0pt,\fontdimen22\textfont2>}
\xymatrix{
[\phantom{x},\phantom{x},\phantom{x}] : F/F_2 \otimes F/F_2 \otimes F/F_2 \ar[r] & F_2/F_3 \otimes F/F_2 \ar[r]^-{[\phantom{x},\phantom{x}]} & F_3/F_4 . 
}
\]
For convenience, we put $w_d=\overline{x_d x_d}=1$. 
Then, for any $j \neq 0$ and $b \neq 0$, 
\[
\overline{[x_i^{-1},x_j^{-1}]} = 
\left\{
\begin{array}{ll}
w_{01} w_j w_{02}^{-1} w_j^{-1} & \hbox{if }i=0, \\
(w_iw_j^{-1})^2 & \hbox{if }i \neq 0, 
\end{array}
\right.
\]
\[
\overline{[x_i,x_j,x_0]} = 
\left\{
\begin{array}{ll}
{[w_j w_{02}w_j^{-1},w_{01}]} & \hbox{ if }i=0,\\
{[(w_iw_j^{-1})^2,w_{01}]} & \hbox{ if }i \neq 0,\\
\end{array}
\right.
\]
and 
\[
\overline{[x_i,x_j,x_b]} \equiv 
\left\{
\begin{array}{lll}
(w_{02}^{-1}w_j^{-1}w_{01}w_j)^2 & \mod{\overline{M}_2} & \hbox{ if }i=0,\\
(w_iw_j^{-1})^{-4} & \mod{\overline{M}_2} & \hbox{ if }i \neq 0. \\
\end{array}
\right.
\]
By (\ref{eqv:rho}), we have 
\begin{align}\label{eqv:rhobar}
\overline{\rho_i} \equiv \overline{\rho_{i,4}} \mod{\overline{F}_3^2 \overline{F}_4 \overline{M}_2 \subset \mathfrak{M}^{(4)}}
\end{align}
with
\begin{align*}
\overline{\rho_{0,4}} =
&\underset{j=1}{\stackrel{d}{\textstyle{\prod}}} 
(w_{01}w_jw_{02}^{-1}w_j^{-1})^{c_{0j}} 
{[w_j w_{02} w_j^{-1}, w_{01}]^{c_{00j}}} 
\\
& \cdot \underset{0<a<b}{\textstyle{\prod}} 
(w_{02}^{-1}w_a^{-1}w_{01}w_a)^{2c_{0a}c_{0b}}
{[(w_aw_b^{-1})^2,w_{01}]^{c_{0ab}}} 
\end{align*}
and 
\begin{align*}
\overline{\rho_{i,4}} =
&\underset{j=1}{\stackrel{d}{\textstyle{\prod}}}
(w_iw_j^{-1})^{2c_{ij}} 
(w_{02}^{-1}w_i^{-1}w_{01}w_i)^{2c_{i0}c_{ij}}
(w_{02}^{-1}w_j^{-1}w_{01}w_j)^{2c_{i0j}}
\\
& \cdot 
(w_{01}w_iw_{02}^{-1}w_i^{-1})^{-c_{i0}} 
{[w_iw_{02}w_i^{-1},w_{01}]^{c_{i0}(c_{i0}-1)/2}}
\\
& \cdot \underset{0<a<b}{\textstyle{\prod}} 
(w_aw_i^{-1})^{-4c_{ia}c_{ib}}
(w_aw_b^{-1})^{-4c_{iab}}
\end{align*}
for $1 \le i \le d$. 
Since 
\begin{align*}
\varPhi \Big(\frac{\partial (w_{01}w_iw_{02}^{-1}w_i^{-1})}{\partial w}\Big)
&\equiv \delta_{w,w_{01}}-\delta_{w,w_{02}}+T\delta_{w,w_i} , 
\\
\varPhi \Big(\frac{\partial [w_i w_{02} w_i^{-1}, w_{01}]}{\partial w}\Big)
&\equiv T \delta_{w,w_{01}}-T \delta_{w,w_{02}}+T^2 \delta_{w,w_i} , 
\\
\varPhi \Big(\frac{\partial (w_{02}^{-1}w_i^{-1}w_{01}w_i)^2}{\partial w}\Big)
&\equiv 2(T+1)\delta_{w,w_{01}}-2(T+1)\delta_{w,w_{02}}+2T \delta_{w,w_i} , 
\\
\varPhi \Big(\frac{\partial [(w_aw_b^{-1})^2,w_{01}]}{\partial w}\Big)
&\equiv 
2T \delta_{w,w_a}
+ 2T \delta_{w,w_b} 
\end{align*}
modulo $(2,T)^3$ for $1 \le i,a,b \le d$, 
a routine calculation using Lemma \ref{lem:linearFox} shows that 
\begin{align}\label{eqv:mod^3}
\varPhi \Big(\frac{\partial \overline{\rho_{i,4}}}{\partial w}\Big)
\equiv q_{w,i}
\mod{(2,T)^3}
\end{align}
with $q_{w,i} \in \varLambda$ such that 
\[
q_{w,0}=
\left\{
\begin{array}{ll}
\underset{j=1}{\stackrel{d}{\textstyle{\sum}}} (c_{0j}+c_{00j}T)
+\underset{0<a<b}{\textstyle{\sum}} 2c_{0a}c_{0b}(T+1) 
& \hbox{if } w=w_{01}, \\
c_{0m}T+c_{00m}T^2
+\underset{b=m+1}{\stackrel{d}{\textstyle{\sum}}} 2(c_{0m}c_{0b}+c_{0mb})T
+\underset{a=1}{\stackrel{m-1}{\textstyle{\sum}}} 2c_{0am}T 
& \hbox{if } w=w_m, 
\end{array}
\right.
\]
and 
\[
q_{w,i}=
\left\{
\begin{array}{ll}
\underset{j=1}{\stackrel{d}{\textstyle{\sum}}} 2(c_{i0}c_{ij}+c_{i0j})(T+1)
-c_{i0}+\frac{c_{i0}(c_{i0}-1)}{2}T
& \hbox{if } w=w_{01}, \\
\underset{j=1}{\stackrel{d}{\textstyle{\sum}}} 2c_{ij}(1+c_{i0}T)
+2c_{i0i}T-c_{i0}T+\frac{c_{i0}(c_{i0}-1)}{2}T^2 & \\
\quad +\underset{0<a<b}{\textstyle{\sum}} 4c_{ia}c_{ib}
+\underset{b=i+1}{\stackrel{d}{\textstyle{\sum}}} 4c_{iib}
+\underset{a=1}{\stackrel{i-1}{\textstyle{\sum}}} 4c_{iai}
& \hbox{if } w=w_i, \\
-2c_{im}+2c_{i0m}T
+\underset{b=m+1}{\stackrel{d}{\textstyle{\sum}}} 4(c_{im}c_{ib}+c_{imb})
+\underset{a=1}{\stackrel{m-1}{\textstyle{\sum}}} 4c_{iam}
& \hbox{if } w=w_m \neq w_i, 
\end{array}
\right.
\]
for $1 \le i \le d$ and $1 \le m \le d-1$. 
Then (\ref{eqv:rhobar}), (\ref{eqv:mod^3}) and Theorem \ref{thm:approx} for $n=4$ yield that 
$E_0(X)+(2,T)^{3-\varepsilon}$ is generated by 
$(2,T)^{3-\varepsilon}$ and $d \times d$ minors of 
the $d \times (d+1)$ matrix $(q_{w,i})_{w,i}$. 
In particular when $d=2$, 
the $2 \times 2$ minors 
\[
\varDelta_0(T)= 
\left|
\begin{array}{cc}
q_{w_{01},0} & q_{w_{01},1} \\
q_{w_1,0} & q_{w_1,1} 
\end{array}
\right|, 
\quad
\varDelta_1(T)= 
\left|
\begin{array}{cc}
q_{w_{01},1} & q_{w_{01},2} \\
q_{w_1,1} & q_{w_1,2} 
\end{array}
\right|, 
\quad
\varDelta_2(T)= 
\left|
\begin{array}{cc}
q_{w_{01},2} & q_{w_{01},0} \\
q_{w_1,2} & q_{w_1,0} 
\end{array}
\right|
\]
of the $2 \times 3$ matrix $(q_{w,i})_{w,i}$ 
satisfy the following congruences modulo $(2,T)^3$; 
\begin{align*}
\varDelta_0(T) \equiv 
&\ 
\big( c_{10}c_{002}+c_{02}{\textstyle \frac{c_{10}(c_{10}-1)}{2}} \big)T^2 
-c_{10}c_{02}T
+2c_{12}(c_{01}+c_{02})
\\
&+ 2\big( c_{12}(c_{10}c_{02}+c_{001}+c_{002})+c_{01}c_{102}+c_{02}c_{101}+c_{10}c_{012} \big)T \\
&+4\big( c_{12}c_{01}c_{02}+c_{112}(c_{01}+c_{02}) \big) , 
\\
%\end{align*}
%\begin{align*}
\varDelta_1(T) \equiv 
&\ 
\big( c_{10}{\textstyle \frac{c_{20}(c_{20}-1)}{2}} + c_{20}{\textstyle \frac{c_{10}(c_{10}-1)}{2}} \big)T^2 
-c_{10}c_{20}T
+2(c_{12}c_{20}+c_{21}c_{10})
\\
&+2\big( 
c_{12}{\textstyle \frac{c_{20}(c_{20}-1)}{2}} + c_{21}{\textstyle \frac{c_{10}(c_{10}-1)}{2}}
+c_{10}c_{20}(c_{12}+c_{21})+c_{10}c_{202}+c_{20}c_{101} 
\big)T \\
&+4\big( 
c_{12}c_{21}(c_{10}+c_{20})
+c_{21}(c_{101}+c_{102})+c_{12}(c_{201}+c_{202})
+c_{10}c_{212}+c_{20}c_{112}
\big) ,
\\
%\end{align*}
%\begin{align*}
\varDelta_2(T) \equiv 
&\ 
\big( c_{20}c_{001}+c_{01}{\textstyle \frac{c_{20}(c_{20}-1)}{2}} \big)T^2 
-c_{01}c_{20}T
+2c_{21}(c_{01}+c_{02})
\\
&+2\big( 
c_{21}(c_{01}c_{20}+c_{001}+c_{002})+c_{01}c_{202}+c_{02}c_{201}+c_{20}c_{012}+c_{01}c_{20}c_{02} 
 \big)T \\
&+4\big( 
c_{21}c_{01}c_{02}+c_{212}(c_{01}+c_{02})
\big) . 
\end{align*}
Then we obtain the following theorems. 

\begin{theorem}\label{thm:d=2:imag}
Assume that $k=\mathbb Q$, $p=2$ and $S=\{\ell_1,\ell_2,\infty\}$. 
Let $\varDelta(T)$ be the Iwasawa polynomial of 
$X=G_{\emptyset}(K^{\mathrm{cyc}})^{\mathrm{ab}}$
for an imaginary quadratic field $K=\mathbb Q(\sqrt{-\ell_1\ell_2})$, 
and let $c_{012}$, $c_{201}$ be $2$-adic integers defined by (\ref{def:ciab}). 
\begin{enumerate}

\item\label{d=2:93}
If $\ell_1 \equiv 9 \pmod{16}$, $\ell_2 \equiv 3 \pmod{8}$ and $\big(\frac{\ell_1}{\ell_2}\big)=1$, then 
\[
\varDelta(T) \equiv {\textstyle T^2 + \big( 1+\big(\frac{2}{\ell_1}\big)_4 \big)T + 2\big( 1-\big(\frac{\ell_2}{\ell_1}\big)_4 \big)} \mod{4\mathbb Z_2 T+8\mathbb Z_2} 
\]
and $c_{012}+c_{201} \equiv \frac{1}{2}\big(1+\big(\frac{2}{\ell_1}\big)_4\big) \pmod{2}$. 

\item\label{d=2:75}
If $\ell_1 \equiv 7 \pmod{16}$, $\ell_2 \equiv 5 \pmod{8}$ and $\big(\frac{\ell_1}{\ell_2}\big)=1$, then 
\[
\varDelta(T) \equiv {\textstyle T^2 + 2\big( 1-\big(\frac{-\ell_1}{\ell_2}\big)_4 \big)} \mod{4\mathbb Z_2 T+8\mathbb Z_2}
\]
and $c_{012} \equiv c_{201} \pmod{2}$. 

\end{enumerate}
\end{theorem}

\begin{proof}
Assume that $\big(\frac{\ell_1}{\ell_2}\big)=1$ and 
either $\ell_1 \equiv 9 \pmod{16}$ and $\ell_2 \equiv 3 \pmod{8}$ 
or $\ell_1 \equiv 7 \pmod{16}$ and $\ell_2 \equiv 5 \pmod{8}$. 
Then, since $X \simeq \mathbb Z_2^2$ as a $\mathbb Z_2$-module (cf.\ \cite{Fer80,Kid79}), 
$E_0(X)$ is a principal ideal of $\varLambda$ generated by 
$\varDelta(T) \equiv T^2 \pmod{2}$ 
(cf.\ \cite[p.299, Example (3)]{Was}). 
By (\ref{eqv:rhobar}), (\ref{eqv:mod^3}) and Theorem \ref{thm:approx}, 
the ideal $E_0(X)+(2,T)^3$ is generated by $(2,T)^3$ and 
$\varDelta_0(T),\varDelta_1(T),\varDelta_2(T)$. 
Hence 
\[
\varDelta_i(T) \equiv u_i(T)\varDelta(T) \mod{(2,T)^3}
\]
with some $u_i(T) \in \varLambda^{\times}$ for each $i \in \{0,1,2\}$. 
By the assumption, we have 
$c_{01}=0$, $c_{02}=1$, $c_{12}=c_{21}=0$, $c_{10} \equiv 2 \pmod{4}$, $c_{20} \equiv 1 \pmod{2}$. 
Then $\frac{c_{10}(c_{10}-1)}{2} \equiv 1 \pmod{2}$. 
Moreover, $(-1)^{c_{001}}=\big(\frac{\ell_1^{\ast}}{2}\big)_4=-1$ 
by Propositions \ref{Redei:ciab} and \ref{Redei:4th}. 
Hence 
\begin{align*}
\varDelta_0(T) &\equiv \varDelta_1(T) \equiv T^2+2(c_{101}+1)T+4c_{112} \mod{(2,T)^3}, \\
\varDelta_2(T) &\equiv T^2+2(c_{012}+c_{201})T+4c_{212} \mod{(2,T)^3}. 
\end{align*}
In particular, $\varDelta_i(T) \in (2,T)^2$ for all $i \in \{0,1,2\}$. 
Since 
\[
T^2 \equiv \varDelta_i(T) \equiv u_i(T)\varDelta(T) \equiv u_i(0)T^2 \mod{(2,T^3)}, 
\]
we have $u_i(T) \in \varLambda^{\times}$, and hence 
\[
\varDelta(T) \equiv u_i(T)^{-1}\varDelta_i(T) \equiv \varDelta_i(T) \mod{(2,T)^3} 
\]
for all $i \in \{0,1,2\}$. 
In particular, $c_{101}+1 \equiv c_{012}+c_{201} \pmod{2}$. 
Since 
\[
(-1)^{c_{101}}=
\left\{
\begin{array}{ll}
\big(\frac{2}{\ell_1}\big)_4
& \hbox{if }\ell_1 \equiv 9 \pmod{16}, \\
\big(\frac{-\ell_1}{2}\big)_4=-1
& \hbox{if }\ell_1 \equiv 7 \pmod{16} 
\end{array}
\right.
\]
and 
\[
(-1)^{c_{112}}=(-1)^{c_{212}}=
\left\{
\begin{array}{ll}
\big(\frac{-\ell_2}{\ell_1}\big)_4=
\big(\frac{\ell_2}{\ell_1}\big)_4 
& \hbox{if }\ell_1 \equiv 9 \pmod{16},\ \ell_2 \equiv 3 \pmod{8}, \\
\big(\frac{-\ell_1}{\ell_2}\big)_4 
& \hbox{if }\ell_1 \equiv 7 \pmod{16},\ \ell_2 \equiv 5 \pmod{8} 
\end{array}
\right.
\]
by Propositions \ref{Redei:ciab} and \ref{Redei:4th}, 
we obtain the statement of Theorem \ref{thm:d=2:imag}. 
\end{proof}

\begin{remark}
Although the R\'edei symbols $[\ell_1,\ell_2,\ell_0]$ and $[\ell_0,\ell_1,\ell_2]$ are not defined in the situation of Theorem \ref{thm:d=2:imag}, 
the congruences for $c_{012}$ and $c_{201}$ imply a certain decomposition law of primes in the $2$-extension $K_{(0,1)}K_{(1,2)}/\mathbb Q$ of degree $32$. 
\end{remark}

\begin{remark}
In the situation of Theorem \ref{thm:d=2:imag}, 
the theorems of Mazur and Wiles (Iwasawa main conjecture, cf.\ \cite{Gre92,MW84} etc.) 
yields the equality $\frac{1}{2}f(T)\varLambda=\varDelta(T)\varLambda$ 
for $f(T) \in 2\varLambda$ such that 
$L_2(s,\chi)=f(\kappa^s-1)$ is the Kubota-Leopoldt $2$-adic $L$-function for 
a quadratic character $\chi$ associated to $\mathbb Q(\sqrt{\ell_1\ell_2})$. 
The construction of $f(T)$ via Stickelberger elements 
induces an algorithm of approximate computation of $\varDelta(T)$. 
On the other hand, 
the proof of Theorem \ref{thm:d=2:imag} does not use these results. 
\end{remark}

\begin{theorem}\label{thm:d=2:real}
Suppose $p=2$, and let $K=\mathbb Q(\sqrt{\ell_1\ell_2})$ be a real quadratic field with prime numbers $\ell_1 \equiv 7 \pmod{16}$, $\ell_2 \equiv 3 \pmod{8}$. 
Then $\widetilde{G}_{\{\infty\}}(K)$ has a minimal presentation 
\[
\entrymodifiers={+!!<0pt,\fontdimen22\textfont2>}
\xymatrix{
1 \ar[r] & \widetilde{R} \ar[r] & \widetilde{H} \ar[r] & \widetilde{G}_{\{\infty\}}(K) \ar[r] & 1
}
\]
where $\widetilde{H}=\langle w_{01},w_1 \rangle$ is a free pro-$2$ group with two generators $w_{01}$, $w_1$, 
and $\widetilde{R}$ is a normal subgroup of $\widetilde{H}$ normally generated by two relations $w_1^2$, $[w_{01},w_1,w_{01}]$. 
\end{theorem}

\begin{proof}
By the assumption, 
$c_{01}=0$, $c_{10} \equiv 2 \pmod{4}$, 
$c_{20} \equiv c_{02}=1 \pmod{2}$ 
and $c_{21}=1-c_{12}$. 
Put $i=2c_{21}$, $j=2c_{12} \in \{0,2\}$. 
Then $\varDelta_i(T) \equiv 2 \pmod{(2,T)^2}$, and hence 
\[
2\varDelta_i(T) \equiv 4 \mod{(2,T)^3}, \quad
T\varDelta_i(T) \equiv 2T \mod{(2,T)^3}. 
\]
Since $\frac{c_{10}(c_{10}-1)}{2} \equiv 1 \pmod{2}$, 
and $(-1)^{c_{001}}=\big(\frac{-\ell_1}{2}\big)_4=-1$ 
by Propositions \ref{Redei:ciab} and \ref{Redei:4th}, 
we have $\varDelta_j(T) \equiv T^2 \pmod{(4,2T)+(2,T)^3}$, 
and hence 
\[
E_0(X)+(2,T)^3 = (2,T^2) 
\]
by (\ref{eqv:rhobar}), (\ref{eqv:mod^3}) and Theorem \ref{thm:approx}. 
Since 
\[
(2,T)^n \subset (2,T)^{n-2}(2,T^2) \subset E_0(X)+(2,T)^{n+1}, 
\]
one can see that 
$(2,T^2) = E_0(X) + (2,T)^n$ 
for any $n \ge 3$ by induction, and hence $E_0(X)=(2,T^2)$. 
Since $X/TX \simeq G_{\{\infty\}}(K)^{\mathrm{ab}}$ is cyclic, 
$X$ is a cyclic $\varLambda$-module, and hence 
\[
X=G_{\{\infty\}}(K^{\mathrm{cyc}})^{\mathrm{ab}} \simeq \varLambda/E_0(X) = \varLambda/(2,T^2) . 
\]
In particular, 
the commutator subgroup of $\widetilde{G}_{\{\infty\}}(K)$ is $G_{\emptyset}(K^{\mathrm{cyc}}(\sqrt{-\ell_1}))$, 
and $G_{\{\infty\}}(K^{\mathrm{cyc}})^{\mathrm{ab}}$ is an abelian group of type $[2,2]$. 
Since 
\begin{align}\label{d=2:real:/T}
G_{\emptyset}(\mathbb Q^{\mathrm{cyc}}(\sqrt{-\ell_1}))^{\mathrm{ab}} \simeq \varLambda/T \simeq \mathbb Z_2
\end{align}
(cf.\ \cite{Fer80,Kid79}), 
the maximal subgroup $G_{\emptyset}(K^{\mathrm{cyc}}(\sqrt{-\ell_1}))$ 
of $G_{\{\infty\}}(K^{\mathrm{cyc}})$ has infinite abelian quotient. 
Therefore $G_{\{\infty\}}(K^{\mathrm{cyc}})$ is a prodihedral pro-$2$ group. 
Then 
\[
\mathrm{Gal}((K^{\mathrm{cyc}})_{\{\infty\}}/K^{\mathrm{cyc}}(\sqrt{-\ell_1}))=G_{\emptyset}(K^{\mathrm{cyc}}(\sqrt{-\ell_1})) \simeq \mathbb Z_2,
\]
and hence $(K^{\mathrm{cyc}})_{\{\infty\}}=K\mathbb Q^{\mathrm{cyc}}(\sqrt{-\ell_1})_{\emptyset}$, which is a $\mathbb Z_2^2$-extension of $K(\sqrt{-\ell_1})$ by (\ref{d=2:real:/T}). 
Recall that the homomorphism $\overline{\varpi}$ gives isomorphisms 
$\widetilde{G}_{\{\infty\}}(K) \simeq \overline{H}/\overline{R}$ 
and $\mathrm{Gal}((K^{\mathrm{cyc}})_{\{\infty\}}/\mathbb Q) \simeq \overline{F}/\overline{R}$. 
Since $w_{01}^{-1}w_{02}=\overline{[x_0,x_d]} \in \overline{F}_2$, 
we have 
$\overline{\varpi}(w_{01}^{-1}w_{02}) \in G_{\emptyset}(K^{\mathrm{cyc}}(\sqrt{-\ell_1}))$, 
and hence $\overline{\varpi}$ induces a minimal presentation 
\[
\entrymodifiers={+!!<0pt,\fontdimen22\textfont2>}
\xymatrix{
1 \ar[r] & \widetilde{H} \cap \overline{R}  \ar[r] & \widetilde{H} \ar[r]^-{\overline{\varpi}} & \widetilde{G}_{\{\infty\}}(K) \ar[r] & 1
} , 
\]
where $\widetilde{H}=\langle w_{01},w_1 \rangle \subset \overline{H}$. 
Then $G_{\{\infty\}}(K^{\mathrm{cyc}})$ is a prodihedral pro-$2$ group 
generated by $\overline{\varpi}(w_1)$ and $\overline{\varpi}([w_{01},w_1])$, 
which has the procyclic maximal subgroup $G_{\emptyset}(K^{\mathrm{cyc}}(\sqrt{-\ell_1}))$ generated by $\overline{\varpi}([w_{01},w_1])$. 
This implies that 
\[
w_1^{-1}[w_{01},w_1]w_1 \equiv [w_{01},w_1]^{-1} \mod{\widetilde{H} \cap \overline{R}} . 
\]
Then $(w_1[w_{01},w_1]^z)^2 \equiv w_1^2 \pmod{\widetilde{H} \cap \overline{R}}$ for any $z \in \mathbb Z_2$, 
and hence $w_1^2 \equiv 1 \pmod{\widetilde{H} \cap \overline{R}}$. 
Moreover, since the maximal subgroup $\widetilde{G}_{\emptyset}(K(\sqrt{-\ell_1}))$ of $\widetilde{G}_{\{\infty\}}(K)$ is an abelian pro-$2$-group generated by 
$\overline{\varpi}(w_{01})$ and $\overline{\varpi}([w_{01},w_1])$, 
we have 
$[w_{01},w_1,w_{01}] \in \widetilde{H} \cap \overline{R}$. 
Therefore $\widetilde{R} \subset \widetilde{H} \cap \overline{R}$, i.e., 
there is a surjective homomorphism $\widetilde{H}/\widetilde{R} \rightarrow \widetilde{G}_{\{\infty\}}(K)$. 
Since $w_1^2 \in \widetilde{R}$ and 
\[
{[w_{01},w_1]w_1^{-1}[w_{01},w_1]w_1 = [w_{01},w_1^2] \in \widetilde{R},} 
\]
there are surjective homomorphisms 
$\mathbb Z_2 \oplus \mathbb Z/2\mathbb Z \rightarrow \widetilde{H}/\widetilde{H}_2\widetilde{R}$ and 
$\mathbb Z_2 \rightarrow \widetilde{H}_2\widetilde{R}/\widetilde{R}$. 
This implies that $\widetilde{H}/\widetilde{R} \simeq \widetilde{G}_{\{\infty\}}(K)$, i.e., $\widetilde{R}=\widetilde{H} \cap \overline{R}$. 
Thus the proof of Theorem \ref{thm:d=2:real} is completed. 
\end{proof}

\begin{remark}
In Theorem \ref{thm:d=2:real}, 
the finiteness of $G_{\emptyset}(K^{\mathrm{cyc}})^{\mathrm{ab}}$ 
(Greenberg's conjecture, cf.\ \cite{Gre76}) 
is certainly verified, using the same description of $\varDelta_i(T) \bmod{(2,T)^3}$ by $c_{ij}$ and $c_{iab}$ as in the case of Theorem \ref{thm:d=2:imag}. 
In fact, $G_{\emptyset}(K^{\mathrm{cyc}})=\{1\}$, i.e., $\widetilde{G}_{\emptyset}(K) \simeq \mathbb Z_2$ in this case (cf.\ e.g.\ \cite{OT97}). 
\end{remark}

\begin{remark}
In Theorem \ref{thm:d=2:imag}, 
we calculated Iwasawa polynomials from a Koch type presentation of $\widetilde{G}_S(\mathbb Q)$. 
Conversely, 
there is a case where an explicit presentation of $\widetilde{G}_{\emptyset}(k)$ is obtained from the Iwasawa polynomial: 
If $p=2$ and $k=\mathbb Q(\sqrt{-\ell_1\ell_2})$ with prime numbers 
$\ell_1 \equiv 7 \pmod{16}$, 
$\ell_2 \equiv 3 \pmod{8}$, 
then $\widetilde{G}_{\emptyset}(k)$
has a minimal presentation 
\[
\entrymodifiers={+!!<0pt,\fontdimen22\textfont2>}
\xymatrix{
1 \ar[r] & R \ar[r] & F \ar[r] & \widetilde{G}_{\emptyset}(k) \ar[r] & 1
}
\]
with a free pro-$2$ group $F$ generated by $\{a,b,c\}$ 
and the normal subgroup $R$ normally generated by 
\[
a^2[a,b], \quad
a^2[b,c,b], \quad
[b,c,a], \quad
[a,c], \quad
[b,c]^{-C_1}[c,b,c]a^{C_1}b^{-C_0}, 
\]
where $C_1, C_0 \in 2\mathbb Z_2$ are the coefficients of the Iwasawa polynomial $\varDelta(T)=T^2+C_1T+C_0$ of $X=G_{\emptyset}(k^{\mathrm{cyc}})^{\mathrm{ab}}$ (cf.\ \cite[Theorem 2.2]{Miz10b}). 
\end{remark}

%%%%%%%%%%%%%%%%%%%%%%%%%%%%%%%%%%%%%%%%%%%%%%%%%%%%%%%%%%%%%%%%%%%%%%%%%%%%%%%
\vspace*{15pt}
\begin{acknowledgements}
The author thanks the referee for helpful comments and suggestion 
for the improvement of this paper.
This work was supported by JSPS KAKENHI Grant Numbers JP26800010, JP17K05167. 
\end{acknowledgements}

%%%%%%%%%%%%%%%%%%%%%%%%%%%%%%%%%%%%%%%%%%%%%%%%%%%%%%%%%%%%%%%%%%%%%%%%%%%%%%%

\vspace*{10pt}

{\footnotesize 
\noindent%\textsc{Yasushi Mizusawa}\\%
Department of Mathematics, Nagoya Institute of Technology, Gokiso, Showa, Nagoya 466-8555, Japan. \\
\texttt{mizusawa.yasushi@nitech.ac.jp}
\\[2mm]%
}

\end{document}